\documentclass[11pt,paper]{article}

\usepackage{amsmath, latexsym, verbatim, amsthm, amsfonts, amssymb}

\newcommand{\vs}[1]{\vskip #1mm}
\newcommand{\X}{\mathbb X}

\newcommand{\R}{\mathbb R}

 \newtheorem{thm}{Theorem}
 \newtheorem{cor}{Corollary}
 \newtheorem{lem}{Lemma}
 \newtheorem{prop}{Proposition}

 \newtheorem{ass}{Assumption}

\begin{document}
\thispagestyle{empty}

\begin{center}
{\bf\large Convergence rates of posterior distributions for observations without the iid structure          }

\end{center}

\vspace{4mm} \sloppy
\begin{center}
{\sc  Yang Xing}\footnote{E-mail address: yang.xing@sekon.slu.se} \\[8pt]
{\it Centre of Biostochastics\\
Swedish University of Agricultural Sciences\\
SE-901 83 Ume\aa, Sweden}\\[8pt]
\end{center}
\vspace{3mm}

\begin{center}
{\large \bf Abstract}

\vspace{4mm}
\begin{minipage}{12cm}
The classical condition on the existence of uniformly exponentially consistent tests for testing the true density against the complement of its arbitrary neighborhood has been widely adopted in study of asymptotics of Bayesian nonparametric procedures. Because we follow a Bayesian approach, it seems to be more natural to explore alternative and appropriate conditions which incorporate the prior distribution. In this paper we supply a new prior-dependent integration condition to establish general posterior convergence rate theorems for observations which may not be independent and identically distributed.
The posterior convergence rates for such observations  have recently studied by Ghosal and van der Vaart \cite{ghv1}. We moreover adopt the Hausdorff $\alpha$-entropy given by Xing and Ranneby \cite{xir1}\cite{xi1}, which is also prior-dependent and smaller than the widely used metric entropies. These lead to extensions of several existing theorems. In particular, we establish a posterior convergence rate theorem for general Markov processes and as its application we improve on the currently known posterior rate of convergence for a nonlinear autoregressive model.

\end{minipage}
\end{center}

\vspace{8mm}


\vs5

\newpage
\setcounter{page}{1} \setcounter{equation}{0}

\section{Introduction}
The aim of this article is to study the asymptotic behavior of posterior distributions based on observations which are not assumed to be independent and identically distributed. Suppose that $\big(\mathfrak{X}^{(n)},\mathcal{A}^{(n)},P_\theta^{(n)}:\,\theta\in \Theta\big)$, $n=1,2,\dots$, are statistical experiments with observations $X^{(n)}$, where the parameter set $\Theta$ does not depend on the index $n$, and suppose that the distributions $P_\theta^{(n)}$ for all $\theta\in \Theta$ admit densities $p^{(n)}_\theta$ relative to a $\sigma$-finite measure $\mu^{(n)}$ on  $\mathfrak{X}^{(n)}$. Denote by $\theta_0$ the true parameter generating the observations $X^{(n)}$. Assume that $P_{\theta}^\infty$ is the infinite product measure $P_\theta^{(1)}P_\theta^{(2)}\cdots P_\theta^{(n)}\cdots$
on the product space $\bigotimes_{n=1}^\infty \mathfrak{X}^{(n)}$. In the sense that each $B\subset \mathfrak{X}^{(n)}$ is identified with the subset $(\mathfrak{X}^{(1)},\mathfrak{X}^{(2)},\dots,\mathfrak{X}^{(n-1)},B,\mathfrak{X}^{(n+1)},\dots)$ of the product space, we have that $P_{\theta}^\infty=P_{\theta}^{(n)}$ holds on $\mathfrak{X}^{(n)}$ for all $n$. In other words, $P_{\theta}^\infty$ is the distribution of the sequence $(X_1,X_2,\dots)$ which makes the observations $X_n$ independent from $P_{\theta}^{(n)}$. Let $d_n$ be a semimetric on $\Theta$. Note that any semimetric $d_n(P_{\theta_1}^{(n)},P_{\theta_2}^{(n)})$ on the space of densities defined on $\mathfrak{X}^{(n)}$ induces naturally a semimetric
$d_n(\theta_1,\theta_2)=d_n(P_{\theta_1}^{(n)},P_{\theta_2}^{(n)})$ on $\Theta$ when the mapping $\theta\mapsto P_{\theta}^{(n)}$ is one-to-one which is assumed in the paper. Given a prior $\Pi_n$ on $\Theta$, the posterior distribution $\Pi_n\bigl( \cdot\,\big|\,X^{(n)}\bigr)$ is a random probability measure given by
$$\Pi_n\bigl( B\,\big|\,X^{(n)}\bigr) ={\int_B\, p_\theta^{(n)}(X^{(n)})\, \Pi_n(d\theta) \over \int_{\Theta}\, p_\theta^{(n)}(X^{(n)})\, \Pi_n(d\theta)}={\int_B\, R_\theta^{(n)}(X^{(n)})\, \Pi_n(d\theta) \over \int_{\Theta}\, R_\theta^{(n)}(X^{(n)})\, \Pi_n(d\theta)}$$
for each measurable subset $B$ in $\Theta$, where $R_\theta^{(n)}(X^{(n)})=p_\theta^{(n)}(X^{(n)})\big/p_{\theta_0}^{(n)}(X^{(n)})$ stands for the likelihood ratio.
Recall that the posterior distribution $\Pi_n(\,\cdot \,|X^{(n)})$ is said to be convergent almost surely at a rate at least $\varepsilon_n$ if there exists $r>0$ such that
$\Pi_n\bigl( \theta\in\Theta:\,d_n(\theta,\theta_0)\geq r\varepsilon_n\,\big|X^{(n)}\bigr)\longrightarrow 0$ almost surely  as $n\to\infty$. Similarly, $\Pi_n(\,\cdot \,|X^{(n)})$ is said to be convergent in probability at a rate at least $\varepsilon_n$ if for any sequence $r_n$ tending to infinity,
$\Pi_n\bigl( \theta\in\Theta:\,d_n(\theta,\theta_0)\geq r_n\varepsilon_n\,\big|X^{(n)}\bigr)\longrightarrow 0$ in probability as $n\to\infty$. Throughout this paper, almost sure convergence and convergence in probability are understood as to be defined with respect to  $P_{\theta_0}^\infty$.

Asymptotics of Bayesian nonparametric procedures has been the focus of a considerable amount of research during past three decades. Much works were concerned with the asymptotic behavior of posterior distributions for i.i.d. observations, see, for instance, Barron, Schervish and Wasserman \cite{bsw}, Ghosal, Ghosh and van der Vaart \cite{ghg}, Shen and Wasserman \cite{shw} and Walker, Lijoi and Prunster \cite{wlp}. Recently, Ghosal and van der Vaart \cite{ghv1} proved several types of posterior convergence rate theorems for non-i.i.d. observations.  Their results reply upon the existence of uniformly exponentially consistent tests, combined with the metric entropy condition and the prior concentration rate. Both the existence of uniformly exponentially consistent tests and the metric entropy condition depend on models, but not on priors. Since the posterior depends on the complexity of the model only through the prior, it is therefore of interest to explore
alternative conditions which incorporate priors. In this paper we use an integration condition together with the Hausdorff $\alpha$-entropy to study convergence rates of posteriors. The integration condition and the Hausdorff $\alpha$-entropy both are prior-dependent.
We show that the integration condition is weaker than the existence of uniformly exponentially consistent tests and holds automatically for an interesting class of metrics used to describe rates of convergence. The latter fact leads to an extension of the results for i.i.d. observations in Walker \cite{wal2}\cite{wal1} and  Xing \cite{xi1}, in which construction of such tests is not necessarily required in order to obtain posterior consistency. The integration condition is moreover useful in construction of priors, as shown when we prove that the convergence rates of the pseudoposteriors given by Walker and Hjort \cite{wah} do not depend on the metric entropy condition.  The Hausdorff $\alpha$-entropy condition was introduced in Xing and Ranneby \cite{xir1}\cite{xi1} and it is weaker than the metric entropy condition. By means of the integration condition and the Hausdorff $\alpha$-entropy, we establish general posterior convergence rate theorems both in the almost sure sense and in the in-probability sense. Particularly, we obtain convergence rate theorems of pseudoposteriors and posteriors for independent observations. We also prove a posterior convergence rate theorem for general Markov chains, which is an extension of a result for stationary $\alpha$-mixing Markov chains given by Ghosal and van der Vaart (\cite{ghv1}, Theorem 5). As applications we improve on the posterior rate of convergence for the nonlinear autoregressive model, see Section 7.4 of Ghosal and van der Vaart \cite{ghv1}. Many authors have studied Bayesian convergence rates for the Gaussian white noise model with a conjugate Gaussian prior (or, equivalently, one has independent normally distributed observations as $N(\theta_i,1/n),\, i=1,2,\dots$ and puts a Gaussian prior independently on $\theta_i,\ i=1,2,\dots n$), see for instance Ghosal and van der Vaart \cite{ghv1}, Scricciolo \cite{scr}, Shen and Wasserman \cite{shw} and Zhao \cite{zha}.
Now by our general posterior convergence rate theorem, we extend their results to multi-normally distributed observations which may not be independent.

The paper is organized as follows. In Section 2 we introduce a prior-dependent integration assumption
and present several different types of general posterior convergence rate theorems. Section 3 contains applications of our general results to independent observations and Markov chains. Section 4 contains concrete applications including nonlinear autoregression model, infinite-dimensional normal model and priors based on uniform distributions.
The technical proofs are collected in Appendix.

Throughout this paper the notation $a \lesssim b$  means $a\leq C b$ for some positive constant $C$ which is universal or fixed in the proof. Write $a \approx b$ if $a \lesssim b$ and $b \lesssim a$. Denote $Pf^\alpha=\int_{\X}f^\alpha dP$ which is the integral of the nonnegative function $f$ with power $\alpha$ relative to the measure $P$ on $\X$.

\section{General Convergence rate theorems}
In this section we introduce a new prior-dependent integration condition to study consistency of posterior distributions. The integration condition
is shown to be automatically fulfilled by a large number of metrics. Together with the Hausdorff $\alpha$-entropy, this integration condition
plays a central roll in our versions of general Bayesian convergence rate theorems.

Let us begin with the following assumption given by Ghosal and van der Vaart \cite{ghv1}, in which they instead equivalently used a constant multiple of the semimetric $e_n$.
\begin{ass}\label{ass:1} Let $K$ be a positive constant. Assume that $\{d_n\}$ and $\{e_n\}$ are two sequences of semimetrics on $\Theta$ such that for every $n$, $\varepsilon>0$ and $\theta_1\in \Theta$ with $d_n(\theta_1,\theta_0)>\varepsilon$, there exists a test $\phi_n$ satisfying
$$P_{\theta_0}^{(n)}\phi_n\leq e^{-Kn\varepsilon^2}\quad{\rm and}\quad \inf\limits_{\theta\in \Theta:\, e_n(\theta,\theta_1)<\varepsilon}P_{\theta}^{(n)}\phi_n\geq 1-e^{-Kn\varepsilon^2}.$$
\end{ass}
Based on Assumption \ref{ass:1}, Ghosal and van der Vaart \cite{ghv1} established a series of general Bayesian convergence rate theorems. Assumption \ref{ass:1} does not depend on the prior distribution. Note that the posterior depends on the complexity of the model only through the prior. As far as the Bayesian approach is concerned, it would be interesting to find some conditions incorporating the prior in study of asymptotic properties. In the following we give such a prior-dependent condition.

Recall that the Hausdorff $\alpha$-entropy $J(\delta,\Theta_1,\alpha,e_n)$ for $\Theta_1\subset \Theta$ is the logarithm of the minimal sum of $\alpha$-th power of prior masses of balls of $e_n$-radius $\leq \delta$ needed to cover $\Theta_1,$ see Xing \cite{xi2} and Xing and Ranneby \cite{xir1} for the details of the Hausdorff $\alpha$-entropy. For simplicity of notations, we define the Hausdorff $\alpha$-constant $C(\delta,\Theta_1,\alpha,e_n):=e^{J(\delta,\Theta_1,\alpha,e_n)}$ of any subset $\Theta_1$ of $\Theta$. Observe that $C(\delta,\Theta_1,\alpha,e_n)$ depends on the prior $\Pi_n$. It was proved in Xing and Ranneby \cite{xir1} that the inequality
$$\Pi_n({\Theta_1})^\alpha\leq C(\delta,\Theta_1,\alpha,e_n)\leq  \Pi_n({\Theta_1})^\alpha\,N(\delta,\Theta_1,e_n)^{1-\alpha}$$
holds for any $0\leq \alpha\leq 1$, where $N(\delta,\Theta_1,e_n)$ denotes the minimal number of balls of $e_n$-radius $\leq \delta$ needed to cover $\Theta_1\subset \Theta.$ Our prior-dependent integration condition is

\begin{ass}\label{ass:2}
Let $\{d_n\}$ and $\{e_n\}$ be two sequences of semimetrics on $\Theta$. For some $\alpha\in (0,1)$ there exist constants $K_1>0$, $K_2>0$ and $K_3\geq 0$ such that the inequality
$$P_{\theta_0}^{(n)}\,\Big(\int_{\theta\in \Theta_1:\,d_n(\theta,\theta_0)> \varepsilon}R_\theta^{(n)}(X^{(n)})\, \Pi_n(d\theta)\Big)^\alpha$$
$$\leq K_1\,e^{-K_2n \varepsilon^2}C(\varepsilon,\{\theta\in \Theta_1:\,d_n(\theta,\theta_0)> \varepsilon\},\alpha,e_n)^{K_3}$$
holds for any $\varepsilon>0$, $\Theta_1\subset\Theta$ and for all $n$ large enough.
\end{ass}

We usually take $K_3=1$ but here we let $K_3\geq 0$ in order to increase the scope of applicability. It was shown in Xing \cite{xi2} that Assumption \ref{ass:2} holds  when the observations are i.i.d. and $r\,e_n=d_n=d$ for some constant $r> 2$ and some metric $d$ which is dominated by the Hellinger distance.
The integral of Assumption \ref{ass:2} depends on the prior $\Pi_n$ and hence is trivially equal to zero when $\Pi_n$ puts zero mass outside of $\theta_0$. So Assumption \ref{ass:2} cannot generally imply Assumption \ref{ass:1}. In fact, Assumption \ref{ass:2} is weaker than Assumption \ref{ass:1} as shown in the following.

\begin{prop}\label{prop:1}
Assumption \ref{ass:1} implies Assumption \ref{ass:2} for all $0<\alpha< 1$, where one can choose $K_1=2$, $K_2=(1-\alpha)\,K\wedge \alpha\,K$ and ${K_3}=1$.
\end{prop}

We shall use the Hellinger distance $ H(f,g)= ||\sqrt{f}-\sqrt{g}||_2$ and its modification
$ H_*(f,g)=\big|\big|(\sqrt{f}-\sqrt{g})\big({2\over 3}\,\sqrt{f\over g}+{1\over 3}\big)^{1/2}||_2$, where $||h||_p=\big(\int_{\mathfrak{X}^{(n)}} |h|^p\,d\mu^{(n)}\big)^{1/p}$.  The inequalities ${1\over \sqrt{3}}\,H(f,g)\leq H_*(f,g)\leq \big|\big|{f/ g}\big|\big|_\infty^{1/4}\,H(f,g)$ hold for all densities $f$ and $g$, since $\big|\big|{f/ g}\big|\big|_\infty\geq 1$.
The quantity $H_*$ was used by  Xing \cite{xi1} in  computation of prior concentration rates. Denote
$$W_n(\theta_0,\varepsilon)=\big\{\,\theta\in \Theta:\, H_*(p_{\theta_0}^{(n)},p_{\theta}^{(n)}\,)\leq \sqrt{{2\over 3}(e^{{3\over 2}n\varepsilon^2}-1)}\,\big\}.$$
Note that $W_n(\theta_0,\varepsilon)$ contains the set
$\big\{\theta\in \Theta:\, H_*(p_{\theta_0}^{(n)},p_{\theta}^{(n)}\,)\leq \sqrt{n}\varepsilon\big\}$ because of $ n\varepsilon^2< {2\over 3}(e^{{3\over 2}n\varepsilon^2}-1)$. The following proposition shows that Assumption \ref{ass:2} holds automatically when $d_n=e_n=d_n^1$ for some metrics $d_n^1$ such that $d_n^1(\theta,\theta_1)^s$ is a convex function of $\theta$ and
$$d_n^1(\theta_1,\theta_2)^2\leq -{2\over n}\log \big(1-{H(p_{\theta_1}^{(n)},p_{\theta_2}^{(n)}\,)^2\over 2}\big)\eqno (1)$$
for all $n$ and $\theta_1,\ \theta_2\in \Theta$, where $s$ is a fixed positive constant. Throughout this paper we let $d_n^1$ stand for a metric with this property.

\begin{prop}\label{prop:2} Let $0<\delta<1/2$ and  $0<\alpha<1$. Then
the inequality
$$P_{\theta_0}^{(n)}\,\Big(\int_{\theta\in \Theta_1:\,d_n^1(\theta,\theta_0)> \varepsilon}R_\theta^{(n)}(X^{(n)})\, \Pi_n(d\theta)\Big)^\alpha$$
$$\leq 2\,e^{-{1\over 2}(1-\alpha)(1-2\delta)^2n \varepsilon^2}C(\delta\,\varepsilon,\{\theta\in \Theta_1:\,d_n^1(\theta,\theta_0)> \varepsilon\},\alpha,d_n^1)$$
holds for all $n$,  $\varepsilon>0$ and $\Theta_1\subset\Theta$.
\end{prop}

Another advantage of adoption of Assumption \ref{ass:2} is that it enables us more easily to construct prior distributions $\Pi_n$ which may receive good posterior convergence rates. Here we present a result which implies that Assumption \ref{ass:2} with $K_3=0$ holds for data-dependent priors $\Pi_n(d\theta)\big/p_\theta^{(n)}(X^{(n)})^{1-\beta}$ for any given constant $0<\beta<1$. Data-dependent priors have been studied by
Wasserman \cite{wa1}, Walker and Hjort \cite{wah} and Xing and Ranneby \cite{xir2}.

\begin{prop} \label{prop:0} The inequality
$$P_{\theta_0}^{(n)}\,\Big(\int_{\theta\in \Theta_1:\,d_n^1(\theta,\theta_0)> \varepsilon}R_\theta^{(n)}(X^{(n)})^\beta\, \Pi_n(d\theta)\Big)^\alpha$$
$$\leq e^{-\bigl((1-\beta)\wedge \beta\bigr)\alpha n \varepsilon^2}\Pi_n(\theta\in \Theta_1:\,d_n^1(\theta,\theta_0)> \varepsilon)^\alpha$$
holds for all $n$, $0<\alpha<1$, $0<\beta<1$, $\varepsilon>0$ and $\Theta_1\subset\Theta$.
\end{prop}
Now we are ready to represent our first main result of this paper.
\begin{thm}\label{thm:1}
Suppose that Assumption \ref{ass:2} holds and that $\varepsilon_n>0$, $n\,\varepsilon_n^2\geq c_0\,\log n$ for all large $n$ and some fixed constant $c_0>0$. Suppose that there exist a constant  $c_1< K_2$ and a sequence of subsets $\Theta_n$ on $\Theta$ such that
$$C(j\varepsilon_n,\{\theta\in\Theta_n: j\varepsilon_n < d_n(\theta,\theta_0)\leq 2j\varepsilon_n\},\alpha,e_n)^{K_3}\leq e^{c_1j^2n\varepsilon_n^2}\, \Pi_n\bigl( W_n(\theta_0,\varepsilon_n) \bigr)^{\alpha}  \eqno (2)$$
for all sufficiently large  integers  $j$  and  $n.$
\smallskip
\noindent Then for each $r$ large enough we have that
$$\Pi_n\bigl(\theta\in \Theta_n:\,d_n(\theta,\theta_0)\geq r\,\varepsilon_n|X^{(n)}\bigr)\longrightarrow 0$$ almost surely as $n\to\infty$. If furthermore there exists $c_2>{1\over c_0}$ such that
$$\sum\limits_{n=1}^\infty{e^{n\,\varepsilon_n^2\,(3+2c_2)}\, \Pi_n(\Theta\setminus \Theta_n)\over \Pi_n\bigl( W_n(\theta_0,\varepsilon_n) \bigr)}<\infty,$$
then there exists a constant $b>0$ such that for each large $r$ and all large $n$,
$$\Pi_n\bigl(\theta\in \Theta:\,d_n(\theta,\theta_0)\geq r\,\varepsilon_n|X^{(n)}\bigr)\leq e^{-bn\varepsilon_n^2}\qquad {\rm almost\ surely}$$
which tends to zero as $n\to\infty$.
\end{thm}
Under Assumption \ref{ass:1} and $\varepsilon_n\gtrsim n^{-\gamma}$ with $0<\gamma<1/2$, Ghosal and van der Vaart (\cite{ghv1}, Theorem 2) proved an almost sure convergence rate theorem and obtained that $P_{\theta_0}^{(n)}\Pi_n\bigl(\theta\in \Theta_n:\,d_n(\theta,\theta_0)\geq r_n\,\varepsilon_n|X^{(n)}\bigr)={\rm O}(\varepsilon_n^2)$ for every $r_n\to\infty$. The upper bound $\varepsilon_n^2$ is slower than   $e^{-bn\varepsilon_n^2}$ of Theorem \ref{thm:1}, and moreover Theorem \ref{thm:1} can be applied to obtain the posterior convergence at the rate $\varepsilon_n=\sqrt{\log n/n}$.
Note that when $K_3=0$ the inequality (2) follows from  $\Pi_n\bigl( W_n(\theta_0,\varepsilon_n) \bigr)\geq e^{-{c_1\over \alpha}n\varepsilon_n^2}$. So Theorem \ref{thm:1} gives that in the special case of $K_3=0$ the concentration rate is precisely equal to the convergence rate. We also mention that in the case that the set $\Theta$ is convex and $d_n(\theta,\theta_0)^s$ for some constant $s>0$ is a bounded convex function of $\theta$ in $\Theta$, it turns out from Jensen's inequality  that the posterior expectation $\hat\theta_n:=\int\theta\,d\Pi_n(\theta|X^{(n)})$ under the assumptions of Theorem \ref{thm:1} yields a point estimator of $\theta_0$ with the convergence rate at least $\varepsilon_n$. Together with Proposition \ref{prop:2}, Theorem \ref{thm:1} implies the following direct consequence for the metric $d_n^1$.

\begin{cor}\label{cor:0}
Suppose that $\varepsilon_n>0$, $n\,\varepsilon_n^2\geq c_0\,\log n$ for all large $n$ and some fixed constant $c_0>0$. Suppose that there exist $0<\alpha<1$, $0<\delta<1/2$ and $c_1< {1\over 2}(1-\alpha)(1-2\delta)^2$ such that
$$C(\delta j\varepsilon_n,\{\theta\in\Theta: j\varepsilon_n < d_n^1(\theta,\theta_0)\leq 2j\varepsilon_n\},\alpha,d_n^1)\leq e^{c_1j^2n\varepsilon_n^2}\, \Pi_n\bigl( W_n(\theta_0,\varepsilon_n) \bigr)^{\alpha}  $$
for all sufficiently large  integers  $j$  and  $n.$
Then there exists a constant $b>0$ such that for each large $r$ and all large $n$,
$$\Pi_n\bigl(\theta\in \Theta:\,d_n^1(\theta,\theta_0)\geq r\,\varepsilon_n|X^{(n)}\bigr)\leq e^{-bn\varepsilon_n^2}\qquad {\rm almost\ surely}$$
which tends to zero as $n\to\infty$.
\end{cor}

It is also worth pointing out that from Lemma 1 in Xing and Ranneby \cite{xir1} it follows that the inequality (2) can be derived from the following two inequalities:

$$N\big(j\varepsilon_n,\{\theta\in\Theta_n: j\varepsilon_n < d_n(\theta,\theta_0)\leq 2j\varepsilon_n\},e_n\big)^{K_3(1-\alpha)}\leq e^{c_3j^2n\varepsilon_n^2}$$
and
$$\Pi_n\big(\theta\in\Theta_n: j\varepsilon_n < d_n(\theta,\theta_0)\leq 2j\varepsilon_n\big)^{K_3\alpha} \leq e^{c_4j^2n\varepsilon_n^2}\, \Pi_n\bigl( W_n(\theta_0,\varepsilon_n) \bigr)^{\alpha} $$
for some constants $c_3$ and $c_4$ with $c_3+c_4<K_2$. Thus, we have the following consequence.

\begin{cor}\label{cor:1}
Suppose that Assumption \ref{ass:2} holds and that $\varepsilon_n>0$, $n\,\varepsilon_n^2\geq c_0\,\log n$ for all large $n$ and some fixed constant $c_0>0$. Suppose that there exist constants  $c_1,\ c_2,\ c_3$ with $c_1(1-\alpha)+c_2\alpha <K_2$ and $c_3>1/c_0$ and there exists a sequence of subsets $\Theta_n$ on $\Theta$ such that for all large $j$  and  $n,$
$$\leqno\quad (i)\quad N\big(j\varepsilon_n,\{\theta\in\Theta_n: j\varepsilon_n < d_n(\theta,\theta_0)\leq 2j\varepsilon_n\},e_n\big)^{K_3}\leq e^{c_1j^2n\varepsilon_n^2};$$
$$\leqno\quad (ii)\quad \Pi_n\big(\theta\in\Theta_n: j\varepsilon_n < d_n(\theta,\theta_0)\leq 2j\varepsilon_n\big)^{K_3}\leq e^{c_2j^2n\varepsilon_n^2}\, \Pi_n\bigl( W_n(\theta_0,\varepsilon_n) \bigr);$$
$$\leqno\quad (iii)\quad \sum\limits_{n=1}^\infty{e^{n\,\varepsilon_n^2\,(3+2c_3)}\, \Pi_n(\Theta\setminus \Theta_n)\over \Pi_n\bigl( W_n(\theta_0,\varepsilon_n) \bigr)}<\infty.$$
\smallskip
\noindent Then there exists a constant $b>0$ such that for each large $r$ and all large $n$,
$$\Pi_n\bigl(\theta\in \Theta:\,d_n(\theta,\theta_0)\geq r\,\varepsilon_n|X^{(n)}\bigr)\leq e^{-bn\varepsilon_n^2}\qquad {\rm almost\ surely}$$
which tends to zero as $n\to\infty$.
\end{cor}

Our next theorem gives another different version of Theorem \ref{thm:1}.

\begin{thm}\label{thm:2}
The following statements are true.

\smallskip
{\rm (a)}  Theorem \ref{thm:1} holds if the inequality {\rm (2)} is replaced by
$$C(\varepsilon_n,\Theta_n,\alpha,e_n)^{K_3}\leq e^{c_1n\varepsilon_n^2}\, \Pi_n\bigl( W_n(\theta_0,\varepsilon_n) \bigr)^{\alpha} \qquad{\rm for\ all\ large\ }n.  $$

{\rm (b)}  Corollary \ref{cor:1} holds if both {\rm (i)} and {\rm (ii)} are replaced by
$$ N(\varepsilon_n,\Theta_n,e_n)^{K_3}\leq e^{c_1n\varepsilon_n^2}\quad{\rm and}\quad \Pi_n(\Theta_n)^{K_3}\leq e^{c_2n\varepsilon_n^2}\, \Pi_n\bigl( W_n(\theta_0,\varepsilon_n) \bigr). $$
\end{thm}

In order to deal with convergence rates of posterior distributions in the sense of in-probability, following Ghosal and van der Vaart \cite{ghv1}, we adopt notations
$V_k(f,g)=\int_{\mathfrak{X}^{(n)}} f\big|\log (f/g)\big|^k\,d\mu^{(n)}$ and
$V_{k,0}(f,g)=\int_{\mathfrak{X}^{(n)}} f\big|\log (f/g)-K(f,g)\big|^k\,d\mu^{(n)}$, where $K(f,g)=\int_{\mathfrak{X}^{(n)}} f\log (f/g)\,d\mu^{(n)}$ is the Kullback-Leibler divergence of densities $f$ and $g$. Denote
$$B_n(\theta_0,\varepsilon;k)=\big\{\theta\in \Theta:\, K(p_{\theta_0}^{(n)},p_{\theta}^{(n)}\,)\leq n\varepsilon^2,\ V_{k,0}(p_{\theta_0}^{(n)},p_{\theta}^{(n)}\,)\leq n^{k/2}\varepsilon^k\big\}.$$
Our result in this direction is

\begin{thm}\label{thm:3}Suppose that Assumption \ref{ass:2} holds and that $k>1$, $\varepsilon_n>0$, $n\,\varepsilon_n^2\geq c_0$ for all large $n$ and some fixed constant $c_0>0$. Suppose that there exist a constant  $c_1< K_2$ and a sequence of subsets $\Theta_n$ on $\Theta$ such that
$$C(j\varepsilon_n,\{\theta\in\Theta_n: j\varepsilon_n < d_n(\theta,\theta_0)\leq 2j\varepsilon_n\},\alpha,e_n)^{K_3}\leq e^{c_1j^2n\varepsilon_n^2}\, \Pi_n\bigl( B_n(\theta_0,\varepsilon_n;k) \bigr)^{\alpha}  \eqno (3)$$
for all sufficiently large  integers  $j$  and  $n.$
\smallskip
\noindent Then for each $r_n\to\infty$ we have that
$$\Pi_n\bigl(\theta\in \Theta_n:\,d_n(\theta,\theta_0)\geq r_n\,\varepsilon_n|X^{(n)}\bigr)\longrightarrow 0$$ in probability as $n\to\infty$. If furthermore there exists $c_2>1$ such that
${e^{c_2n\varepsilon_n^2}\, \Pi_n(\Theta\setminus \Theta_n)\over \Pi_n\bigl( B_n(\theta_0,\varepsilon_n;k)  \bigr)}\longrightarrow 0$ as $n\to\infty$,
then $$\Pi_n\bigl(\theta\in \Theta:\,d_n(\theta,\theta_0)\geq r_n\,\varepsilon_n|X^{(n)}\bigr)\longrightarrow 0$$
 in probability as $n\to\infty$.
\end{thm}

Similarly, Theorem \ref{thm:3} holds if one replaces the inequality (3) by
$$C(\varepsilon_n,\Theta_n,\alpha,e_n)^{K_3}\leq e^{c_1n\varepsilon_n^2}\, \Pi_n\bigl(B_n(\theta_0,\varepsilon_n;k)  \bigr)^{\alpha}\quad {\rm for\ large}\ n.  $$
Moreover, as a consequence of Theorem \ref{thm:3} we obtain the following result which is a slightly stronger version of Theorem 1 in Ghosal and van der Vaart \cite{ghv1}.

\begin{cor}\label{cor:2}
Suppose that Assumption \ref{ass:2} holds and that $k>1$, $\varepsilon_n>0$, $n\,\varepsilon_n^2\geq c_0$ for all large $n$ and some fixed constant $c_0>0$. Suppose that there exist constants $c_1,\ c_2>0$ with $c_1(1-\alpha)+c_2\alpha <K_2$, $c_3>1$ and a sequence of subsets $\Theta_n$ on $\Theta$ such that for all large  $j$  and  $n,$
$$\leqno\quad (i)\quad N\big(j\varepsilon_n,\{\theta\in\Theta_n: j\varepsilon_n < d_n(\theta,\theta_0)\leq 2j\varepsilon_n\},e_n\big)^{K_3}\leq e^{c_1j^2n\varepsilon_n^2};$$
$$\leqno \quad (ii)\quad \Pi_n\big(\theta\in\Theta_n: j\varepsilon_n < d_n(\theta,\theta_0)\leq 2j\varepsilon_n\big)^{K_3} \leq e^{c_2j^2n\varepsilon_n^2}\, \Pi_n\bigl(B_n(\theta_0,\varepsilon_n;k) \bigr); $$
$$\leqno \quad (iii)\quad {e^{c_3n\varepsilon_n^2}\, \Pi_n(\Theta\setminus \Theta_n)\over \Pi_n\bigl( B_n(\theta_0,\varepsilon_n;k)  \bigr)}\longrightarrow 0\qquad {\rm as}\quad n\to\infty.$$
\smallskip
\noindent Then for each $r_n\to\infty$ we have that
$$\Pi_n\bigl(\theta\in \Theta:\,d_n(\theta,\theta_0)\geq r_n\,\varepsilon_n|X^{(n)}\bigr)\longrightarrow 0$$ in probability as $n\to\infty$.
\end{cor}

\section{Some Special Cases}
In this section we apply our general convergence rate theorems to i.n.i.d. observations and Markov processes. For i.n.i.d. observations we establish almost sure convergence rate theorems both on pseudoposterior distributions and on posterior distributions. We derive an almost sure posterior convergence rate theorem for general Markov processes.

\subsection{Independent observations} We consider the case that $X^{(n)}$ is a random vector $(X_1,X_2,\dots,X_n)$ of independent variables $X_i$, where each $X_i$ is generated from  some density $p_{\theta,i}$ relative to a $\sigma$-finite measure $\mu_i$ on $(\mathfrak{X}_i,\mathcal{A}_i)$, and that $P_\theta^{(n)}$ is the product distribution with the density $p_\theta^{(n)}(X^{(n)})=\prod_{i=1}^np_{\theta,i}(x_i)$ relative to the direct product measure $\mu^{(n)}=\mu_1\times\mu_2\times\dots\times\mu_n$ on $\mathfrak{X}^{(n)}=\mathfrak{X}_1\times \mathfrak{X}_2\times\dots\times \mathfrak{X}_n$. Assume that $d_n^0(\theta_1,\theta_2)=\big({1\over n}\sum_{i=1}^nH_i(p_{\theta_1,i},p_{\theta_2,i})^2\big)^{1/2}$, where each $H_i(p_{\theta_1,i},p_{\theta_2,i})=\big(\int(\sqrt{p_{\theta_1,i}}
-\sqrt{p_{\theta_2,i}})^2\,d\mu_i\big)^{1/2}$ is the Hellinger diatance between $p_{\theta_1,i}$ and $p_{\theta_2,i}$ relative to $\mu_i$ on $\mathfrak{X}_i$. It is clear that $d_n^0$ satisfies the triangle inequality and hence is a metric on $\Theta$. Denote $H_{*,i}(p_{\theta_1,i},p_{\theta_2,i})=\big(\int(\sqrt{p_{\theta_1,i}}
-\sqrt{p_{\theta_2,i}})^2({2\over 3}\sqrt{p_{\theta_1,i}\over p_{\theta_2,i}} +{1\over 3}) \,d\mu_i\big)^{1/2}.$
An advantage of adoption of $H_*$ in computation of concentration rates for independent observations is that we have the following quality
$$1+{3\over 2}\,H_*\Big(\prod_{i=1}^np_{\theta_1,i},\prod_{i=1}^np_{\theta_2,i}\Big)^2=\prod_{i=1}^n \big(1+ {3\over 2}\,H_{*,i}(p_{\theta_1,i},p_{\theta_2,i})^2\big)$$
$$\leq e^{{3\over 2}\,\sum_{i=1}^n H_{*,i}(p_{\theta_1,i},p_{\theta_2,i})^2},$$
which implies that $W_n(\theta_0,\varepsilon)$ contains the set
$$\overline{W}_n(\theta_0,\varepsilon):=\big\{\theta\in \Theta:\, {1\over n}\sum_{i=1}^nH_{*,i}(p_{\theta_0,i},p_{\theta,i})^2\leq \varepsilon^2\big\}.$$
Similarly, we have
$$1-{1\over 2}\,H\Big(\prod_{i=1}^np_{\theta_1,i},\prod_{i=1}^np_{\theta_2,i}\Big)^2=\prod_{i=1}^n \big(1- {1\over 2}\,H_i(p_{\theta_1,i},p_{\theta_2,i})^2\big)$$
$$\leq e^{-{1\over 2}\,\sum_{i=1}^n H_i(p_{\theta_1,i},p_{\theta_2,i})^2}= e^{-{1\over 2}\,n\,d_n^0(\theta_1,\theta_2)^2},$$
which implies that the metric $d_n^0$ satisfies the inequality (1) and hence by the convexity of $(d_n^0)^2$ one can apply Proposition \ref{prop:2} and Proposition  \ref{prop:0} for  $d_n^0$.
Now we are ready to present two results for i.n.i.d. observations by means of $\overline{W}_n(\theta_0,\varepsilon)$ and $d_n^0$.
\bigskip

\noindent{\it 3.1.1 Pseudoposterior Convergence Rate.}  Given $0<\beta< 1$, we define a pseudoposterior distribution $\Pi_{\beta,n}$ based on the prior $\Pi_n$ by
$$\Pi_{\beta,n}\bigl( B\,\big|\,X_1,X_2,\dots,X_n\bigr) ={\int_B\prod\limits_{i=1}^np_{\theta,i}(X_i)^\beta\, \Pi_n(d\theta) \over \int_{\Theta}\prod\limits_{i=1}^np_{\theta,i}(X_i)^\beta\, \Pi_n(d\theta)}\qquad{\rm for\ each}\quad B\subset \Theta.$$
In other words, we use the data-dependent prior
$\Pi_n(d\theta)\big/\prod\limits_{i=1}^np_{\theta,i}(X_i)^{1-\beta}$.
Wasserman \cite{wa1} first applied  psuedolikelihood function-data-dependent priors in study of asymptotic inference for mixture models. The pseudoposterior $\Pi_{\beta,n}$ for i.i.d. observations was introduced by Walker and Hjort \cite{wah} who proved a Hellinger consistency theorem when $\beta=1/2$. The Hellinger consistency theorem for any $0<\beta<1$ was obtained by Xing and Ranneby \cite{xir2}. Here we study the convergence rates of the pseudoposteriors for i.n.i.d. observations. Using Proposition \ref{prop:0} for  $d_n^0$, we obtain

\begin{prop} \label{prop:3} The inequality
$$P_{\theta_0}^{(n)}\,\Big(\int_{\theta\in \Theta_1:\,d_n^0(\theta,\theta_0)> \varepsilon}\Bigl(\prod\limits_{i=1}^n{p_{\theta,i}(X_i)\over p_{\theta_0,i}(X_i)}\Bigr)^\beta\, \Pi_n(d\theta)\Big)^\alpha$$
$$\leq e^{-\bigl((1-\beta)\wedge \beta\bigr)\alpha n \varepsilon^2}\Pi_n(\theta\in \Theta_1:\,d_n^0(\theta,\theta_0)> \varepsilon)^\alpha$$
holds for all $n$, $0<\alpha<1$, $0<\beta<1$, $\varepsilon>0$ and $\Theta_1\subset\Theta$.
\end{prop}

Therefore, we have

\begin{thm}\label{thm:4.1}
Let $0<\beta<1$. Suppose that $\varepsilon_n>0$, $n\,\varepsilon_n^2\geq c_0\,\log n$ for all large $n$ and some fixed constant $c_0>0$. Suppose that there exists $c_1>0$ such that
$$\Pi_n(\theta\in \Theta:\,d_n^0(\theta,\theta_0)> \varepsilon_n)\leq e^{c_1n\varepsilon_n^2}\, \Pi_n\bigl( \overline{W}_n(\theta_0,\varepsilon_n) \bigr)  $$
for all large $n$. Then for each large $r$,
$$\Pi_{\beta,n}\bigl(\theta\in \Theta:\,d_n^0(\theta,\theta_0)\geq r\,\varepsilon_n|\,X_1,X_2,\dots,X_n\bigr)\longrightarrow 0$$
almost surely as $n\to\infty$.
\end{thm}

Since the total mass of $\Pi_n$ is always equal to one, Theorem \ref{thm:4.1} implies that the convergence rate $\varepsilon_n$ of the pseudoposterior distribution $\Pi_{\beta,n}$ can be completely determined  by the concentration condition $\Pi_n\bigl( \overline{W}_n(\theta_0,\varepsilon_n) \bigr)\geq e^{c_1n\varepsilon_n^2}.$ In other words, the convergence rate does not depend on the rate of the metric entropy which describes how large the model is.

\bigskip

\noindent{\it 3.1.2 Posterior Convergence Rate.} By a result of Birg$\acute{{\rm e}}$ (see \cite{lec1}, page 491, or \cite{ghv1}, Lemma 2) we know that there exist tests satisfying Assumption \ref{ass:1}. Based on this fact, Ghosal and van der Vaart (\cite{ghv1}, Theorem 4) gave an in-probability convergence rate theorem for i.n.i.d. observations and the metric $d_n^0$. Now, together with Proposition \ref{prop:2} and $\overline{W}_n(\theta_0,\varepsilon)\subset W_n(\theta_0,\varepsilon)$, Theorem \ref{thm:1} implies the following almost sure assertion.

\begin{thm}\label{thm:4} Let $0<\delta<1/2$ and $0<\alpha<1$.
Suppose that $\varepsilon_n>0$, $n\,\varepsilon_n^2\geq c_0\,\log n$ for all large $n$ and some fixed constant $c_0>0$. Suppose that there exist $c_1< {1\over 2} (1-\alpha)(1-2\delta)^2$, $c_2>{1\over c_0}$ and a sequence of subsets $\Theta_n$ on $\Theta$ such that
$$C(\delta j\varepsilon_n,\{\theta\in\Theta_n: j\varepsilon_n < d_n^0(\theta,\theta_0)\leq 2j\varepsilon_n\},\alpha,d_n^0)\leq e^{c_1j^2n\varepsilon_n^2}\, \Pi_n\bigl( \overline{W}_n(\theta_0,\varepsilon_n) \bigr)^{\alpha}  $$
for all large  $j$, $n$, and
$$\sum\limits_{n=1}^\infty{e^{n\,\varepsilon_n^2\,(3+2c_2)}\, \Pi_n(\Theta\setminus \Theta_n)\over \Pi_n\bigl( \overline{W}_n(\theta_0,\varepsilon_n) \bigr)}<\infty.$$
\smallskip
\noindent Then there exists $b>0$ such that for each large $r$ and all large $n$,
$$\Pi_n\bigl(\theta\in \Theta:\,d_n^0(\theta,\theta_0)\geq r\,\varepsilon_n|X^{(n)}\bigr)\leq e^{-bn\varepsilon_n^2}\qquad {\rm almost\ surely}.$$
\end{thm}

For readers' convenience, we here copy a direct consequence of Theorem \ref{thm:4} for $\alpha=1/2$.

\begin{cor}\label{cor:3}
Let $0<\delta<1/2$. Suppose that $\varepsilon_n>0$, $n\,\varepsilon_n^2\geq c_0\,\log n$ for all large $n$ and some fixed constant $c_0>0$. Suppose that there exist $c_1,\ c_2,\ c_3$ with $c_1+c_2<{1\over 2}(1-2\delta)^2$ and $c_3>1/c_0$ and a sequence of subsets $\Theta_n$ on $\Theta$ such that for all large $j$  and  $n,$
$$\leqno\quad (i)\quad N\big(\delta j\varepsilon_n,\{\theta\in\Theta_n: j\varepsilon_n < d_n^0(\theta,\theta_0)\leq 2j\varepsilon_n\},d_n^0\big)\leq e^{c_1j^2n\varepsilon_n^2};$$
$$\leqno\quad (ii)\quad \Pi_n\big(\theta\in\Theta_n: j\varepsilon_n < d_n^0(\theta,\theta_0)\leq 2j\varepsilon_n\big)\leq e^{c_2j^2n\varepsilon_n^2}\, \Pi_n\bigl( \overline{W}_n(\theta_0,\varepsilon_n) \bigr);$$
$$\leqno\quad (iii)\quad \sum\limits_{n=1}^\infty{e^{n\,\varepsilon_n^2\,(3+2c_3)}\, \Pi_n(\Theta\setminus \Theta_n)\over \Pi_n\bigl( \overline{W}_n(\theta_0,\varepsilon_n) \bigr)}<\infty.$$
\smallskip
\noindent Then there exists $b>0$ such that for each large $r$ and all large $n$,
$$\Pi_n\bigl(\theta\in \Theta:\,d_n^0(\theta,\theta_0)\geq r\,\varepsilon_n|X^{(n)}\bigr)\leq e^{-bn\varepsilon_n^2}\qquad {\rm almost\ surely}.$$
\end{cor}

\subsection{Markov chains}
 Let $X_0,X_1,\dots$ be a Markov chain with transition density $p_\theta(y|x)$ and initial density $q_\theta(x_0)$ with respect to some $\sigma$-finite measure $\mu$ on a measurable space $(\mathfrak{X}, {\cal A}).$ Here we assume that for each $\theta\in \Theta$ the 2-variable function $(x,y)\mapsto p_\theta(y|x)$ is measurable. So the joint distribution $P_{\theta}^{(n)}$ of $X_0,X_1,\dots,X_n$ has a density given by $p_\theta^{(n)}(x^{(n)})=q_\theta(x_0)\prod\limits_{i=1}^n p_\theta(x_i|x_{i-1})$ relative to the product measure $\mu(x_0)\mu(x_1)\dots\mu(x_n)$. We shall adopt the following Hellinger type semimetrics.
$$H\big(p_{\theta_1}(y|x),p_{\theta_2}(y|x)\big)=\Big(\int_{\mathfrak{X}}\int_{\mathfrak{X}}\big(\sqrt{p_{\theta_1}(y|x)}
-\sqrt{p_{\theta_2}(y|x)}\,\big)^2 \,d\mu(y)d\nu(x)\Big)^{1/2},$$
$$H\big(q_{\theta_1}(x),q_{\theta_2}(x)\big)=\Big(\int_{\mathfrak{X}}\big(\sqrt{q_{\theta_1}(x)}
-\sqrt{q_{\theta_2}(x)}\,\big)^2 \,d\mu(x)\Big)^{1/2},$$
$$H_*\big(p_{\theta_1}(y|x),p_{\theta_2}(y|x)\big)$$$$=\Big(\int_{\mathfrak{X}}\int_{\mathfrak{X}}\big(\sqrt{p_{\theta_1}(y|x)}
-\sqrt{p_{\theta_2}(y|x)}\,\big)^2\big({2\over 3}\sqrt{p_{\theta_1}(y|x)\over p_{\theta_2}(y|x)} +{1\over 3}\big) \,d\mu(y)d\nu(x)\Big)^{1/2},$$
$$H_*\big(q_{\theta_1}(x),q_{\theta_2}(x)\big)=\Big(\int_{\mathfrak{X}}\big(\sqrt{q_{\theta_1}(x)}
-\sqrt{q_{\theta_2}(x)}\,\big)^2\big({2\over 3}\sqrt{q_{\theta_1}(x)\over q_{\theta_2}(x)} +{1\over 3}\big) \,d\mu(x)\Big)^{1/2}.$$
Denote
$${W}^1_n(\theta_0,\varepsilon)=\big\{\theta\in \Theta:\, H_*(p_{\theta_0},p_{\theta})^2+{1\over n}H_*(q_{\theta_0},q_{\theta})^2\leq \varepsilon^2\big\}.$$
By means of the metric $d(\theta,\theta_0):=H(p_{\theta},p_{\theta_0})$, Ghosal and van der Vaart (\cite{ghv1}, Theorem 5) gave an in-probability posterior convergence rate theorem for stationary $\alpha$-mixing Markov chains. Since calculation of the $\alpha$-mixing coefficients is generally not easy and many processes are neither mixing nor stationary, it seems worth to develop a posterior convergence rate theorem for Markov chains which may be neither stationary nor $\alpha$-mixing. Now we have an almost sure assertion in this direction. Our result is based on the following proposition.

\begin{prop}\label{prop:4} Suppose that there exist a $\mu$-integrable function $r(y)$ and  constants $a_1\geq  a_0>0$  with $a_1\geq  1$ such that $d\nu(y)=r(y)d\mu(y)$ and
$a_0r(y)\leq p_{\theta}(y|x)\leq  a_1 r(y)$
for all $\theta\in\Theta$ and $x,y\in \mathfrak{X}$. Let $0<\delta<{\sqrt{a_0}\over 2\sqrt{a_1}}$ and $0<\alpha<{1\over 2}$.  Then
the inequality
$$P_{\theta_0}^{(n)}\,\Big(\int_{\theta\in \Theta_1:\,d(\theta,\theta_0)> \varepsilon} {q_\theta(X_0)\over q_{\theta_0}(X_0)}\prod\limits_{i=1}^n{p_\theta(X_i|X_{i-1})\over p_{\theta_0}(X_i|X_{i-1})}\, \Pi_n(d\theta)\Big)^\alpha$$
$$\leq 2\,e^{-({1\over 2}-\alpha)({\sqrt{a_0}\over 2}-\sqrt{a_1}\delta)^2n \varepsilon^2}C(\delta\,\varepsilon,\{\theta\in \Theta_1:\,d(\theta,\theta_0)> \varepsilon\},\alpha,d)$$
holds for all $n$,  $\varepsilon>0$ and $\Theta_1\subset\Theta$, where $d(\theta,\theta_0)=H(p_{\theta},p_{\theta_0})$.
\end{prop}

Therefore we have

\begin{thm}\label{thm:7}  Suppose that all assumptions of Proposition \ref{prop:4} hold and
suppose that $\varepsilon_n>0$, $n\,\varepsilon_n^2\geq c_0\,\log n$ for all large $n$ and some fixed constant $c_0>0$. Suppose that there exist $c_1< ({1\over 2}-\alpha)({\sqrt{a_0}\over 2}-\sqrt{a_1}\delta)^2$, $c_2>{1\over c_0}$ and a sequence of subsets $\Theta_n$ on $\Theta$ such that
$$C({ \delta j\varepsilon_n},\{\theta\in\Theta_n: j\varepsilon_n < d(\theta,\theta_0)\leq 2j\varepsilon_n\},\alpha,d)\leq e^{c_1j^2n\varepsilon_n^2}\, \Pi_n\bigl( W^1_n(\theta_0,\varepsilon_n) \bigr)^{\alpha}  $$
for all large  $j$, $n$, and
$$\sum\limits_{n=1}^\infty{e^{n\,\varepsilon_n^2\,(3 a_1+4c_2)}\, \Pi_n(\Theta\setminus \Theta_n)\over \Pi_n\bigl( W^1_n(\theta_0,\varepsilon_n) \bigr)}<\infty.$$
\smallskip
\noindent Then there exists $b>0$ such that for each large $r$ and all large $n$,
$$\Pi_n\bigl(\theta\in \Theta:\,d(\theta,\theta_0)\geq r\,\varepsilon_n|\,X_0,X_1,\dots,X_n\bigr)\leq e^{-bn\varepsilon_n^2}\qquad {\rm almost\ surely}.$$
\end{thm}

By choosing $\delta={\sqrt{a_0}\over 4\sqrt{a_1}}$ and $\alpha={1\over 4}$ we can easily get

\begin{cor}\label{cor:4}
Suppose that there exist a $\mu$-integrable function $r(y)$ and  constants $a_1\geq  a_0>0$  such that $d\nu(y)=r(y)d\mu(y)$ and
$a_0r(y)\leq p_{\theta}(y|x)\leq  a_1 r(y)$
for all $\theta\in\Theta$ and $x,y\in \mathfrak{X}$.
Suppose that $\varepsilon_n>0$, $n\,\varepsilon_n^2\geq c_0\,\log n$ for all large $n$ and some fixed constant $c_0>0$. Suppose that there exist $c_1,\ c_2,\ c_3$ with $3c_1+c_2<a_0/16$ and $c_3>1/c_0$ and a sequence of subsets $\Theta_n$ on $\Theta$ such that for all large $j$  and  $n,$
$$\leqno\quad (i)\quad N\big({\sqrt{a_0}\over 4\sqrt{a_1}}j\varepsilon_n,\{\theta\in\Theta_n: j\varepsilon_n < d(\theta,\theta_0)\leq 2j\varepsilon_n\},d\big)\leq e^{c_1j^2n\varepsilon_n^2};$$
$$\leqno\quad (ii)\quad \Pi_n\big(\theta\in\Theta_n: j\varepsilon_n < d(\theta,\theta_0)\leq 2j\varepsilon_n\big)\leq e^{c_2j^2n\varepsilon_n^2}\, \Pi_n\bigl( W^1_n(\theta_0,\varepsilon_n) \bigr);$$
$$\leqno\quad (iii)\quad \sum\limits_{n=1}^\infty{e^{n\,\varepsilon_n^2\,(3 a_1+4c_3)}\, \Pi_n(\Theta\setminus \Theta_n)\over \Pi_n\bigl( W^1_n(\theta_0,\varepsilon_n) \bigr)}<\infty.$$
\smallskip
\noindent Then there exists $b>0$ such that for each large $r$ and all large $n$,
$$\Pi_n\bigl(\theta\in \Theta:\,d(\theta,\theta_0)\geq r\,\varepsilon_n|\,X_0,X_1,\dots,X_n\bigr)\leq e^{-bn\varepsilon_n^2}\qquad {\rm almost\ surely}.$$
\end{cor}

\section{Applications}
In this section we gives three examples of applications of our theorems. By means of Corollary \ref{cor:4}, we improve on the posterior rate of convergence for the nonlinear autoregressive model in  Ghosal and van der Vaart \cite{ghv1}. Corollary \ref{cor:0} is applied to find the posterior convergence rate for an infinite-dimensional normal model, which extends the known results in Ghosal and van der Vaart \cite{ghv1}, Scricciolo \cite{scr}, Shen and Wasserman \cite{shw} and Zhao \cite{zha} for the white noise model with a conjugate prior. Finally, we use Corollary \ref{cor:3} to study priors based on uniform distributions, which extends the corresponding result for priors based on discrete distributions in  Ghosal and van der Vaart \cite{ghv1}.
\bigskip

\noindent{\it 4.1. Nonlinear autoregression.} We observe $X_1,X_2,\dots,X_n$ of a time series $\{X_t:t\in\mathbb{Z}\}$ given by
$$X_i=f(X_{i-1})+\bar\varepsilon_i\qquad {for}\quad i=1,2,\dots,n,$$
where $\bar\varepsilon_1,\bar\varepsilon_2,\dots,\bar\varepsilon_n$ are i.i.d. random variables with the standard normal distribution and the unknown regression function $f$ is in the space ${\cal F}$ which consists of all functions $f$ with $\sup\limits_{x\in \R}\big|f(x)\big|\leq M$ for some fixed positive constant $M$. Let $q_f(x)$ be the density of $X_0$ relative to the Lebesgue measure  $d\mu$  on $\R$. So $X_0,X_1,\dots $ can be considered as a Markov chain generated by the transition density $p_f(y|x)=\phi\big(y-f(x)\big)$ with $\phi(x)=(2\pi)^{-1/2}e^{-x^2/2}$ and the initial density $q_f(x)$. Since $\phi(x)$ is a strictly positive continuous function tending to zero as $x\to \pm\infty$, there exist two constants $0<a_0<1<a_1$ depending only on $M$ such that $a_0\phi(y)\leq p_f(y|x)\leq a_1\phi(y)$ for all $f\in{\cal F}$ and $-\infty<y,\,x<\infty.$ Assume that there exists a constant $N>0$ such that the set of initial densities of the Markov chain satisfies $H_*(q_{f_1},q_{f_2})\leq N$ for all initial densities $q_{f_1}$ and $q_{f_2}$. For instance, all of the initial densities with $a_0\phi(x)\leq q_f(x)\leq a_1\phi(x)$ satisfy $H_*(q_{f_1},q_{f_2})\leq \sqrt{2}(a_1/a_0)^{1/4}$ and hence form a set with the requirement. Define a measure $d\nu=\phi d\mu$ in $\R$ and a norm $||f||_2=\big(\int_{\R}|f|^2 d\nu\big)^{1/2}$ on ${\cal F}$. Assume that the true regression function $f_0\in{\cal F}$ belongs to the Lipschitz continuous space $Lip_M$, which consists of all functions $f$ on $(-\infty,\infty)$ satisfying $|f(x)-f(y)|\leq L\,|x-y|$ for all $-\infty<x,\,y<\infty$, where $L$ is a fixed positive constant. When the Markov chain is stationary,  Ghosal and van der Vaart (\cite{ghv1}, Section 7.4) constructed a prior on the regression functions and obtained the in-probability posterior convergence rate $n^{-1/3}(\log n)^{1/2}$, which is the minimax rate times the logarithmic factor $(\log n)^{1/2}$.
In the following we shall apply Corollary \ref{cor:4} to get the posterior convergence rate $n^{-1/3}(\log n)^{1/6}$ in the almost sure sense for a general Markov chain defined as above.

First, we note that for any $f\in {\cal F}$,
$$H_*(p_{f_0},p_f)^2+{1\over n}\,H_*(q_{f_0},q_f)^2\leq \sqrt{a_1\over a_0}H(p_{f_0},p_f)^2+{N^2\over n}$$
$$={1\over 2}\sqrt{a_1\over a_0}\int_{-\infty}^\infty\Big(1-e^{-{(f(x)-f_0(x))^2\over 4}}\Big)\,d\nu(x)+{N^2\over n}\leq {||f-f_0||_2^2\over 8}\sqrt{a_1\over a_0}+{N^2\over n},$$
where the last inequality follows from the elementary inequality $1-e^{-t}\leq t$. Hence for some small constant $b_1>0$ we have that
$W_n^1(f_0,\varepsilon_n)\supset \{f\in{\cal F}: ||f-f_0||_2\leq b_1\,\varepsilon_n\}$ for all large $n$. Similarly, $||f-f_0||_2\approx H(p_f,p_{f_0})$  hold for all $f\in {\cal F}$ with $||f-f_0||_2\leq 1$. Hence Corollary \ref{cor:4} works well for the metric $||\cdot||_2$.

We also need some basic facts on approximation of Lipschitz continuous functions by means of step functions.
Given a finite interval $[-A_n,A_n)$ and a positive integer $K_n$, we make the partition $[-A_n,A_n)=\bigcup_{k=1}^{K_n}I_k$ with $I_k=\big[-A_n+ {2A_n(k-1)\over K_n},-A_n+ {2A_nk\over K_n}\big)$ for $k=1,2,\dots,K_n$. Write $I_0=\R\setminus [-A_n,A_n)$. The space of step functions relative to the partition is the set of functions $h:[-A_n,A_n)\mapsto \R$ such that $h$ is identically equal to some constant on each $I_k$ for $k=1,2,\dots,K_n$, more precisely, $h(x)=\sum_{k=1}^{K_n}\beta_k\,1_{I_k}(x)$ for some $\beta=(\beta_1,\beta_2,\dots,\beta_{K_n})\in [-M,M]^{K_n}\subset \R^{K_n}$, where $1_{I_k}(x)$ denotes the indicator function of $I_k$.
Denote by $f_\beta(x)$ the function on $(-\infty,\infty)$ which is equal to $\sum_{k=1}^{K_n}\beta_k\,1_{I_k}(x)$ on $[-A_n,A_n)$ and vanish outside $[-A_n,A_n)$. Hence $f_\beta\in {\cal F}$ and $||f_{\beta_1}-f_{\beta_2}||_2=||\beta_1-\beta_2||_*$, where $||\beta||_*=\big(\sum_{k=1}^{K_n}\beta_k^2(\int_{I_k}d\nu)^2\big)^{1/2}$. Let $\Pi_n$  be the prior on ${\cal F}$ which is induced by the map $\beta\mapsto f_\beta$ such that all the coordinates $\beta_k$ of $\beta$ are chosen to be i.i.d. random variables with the uniform distribution on $[-M,M]$. Hence the support ${\cal F}_n$ of $\Pi_n$ consists of all such functions $f_\beta$. Take $A_n=2\sqrt{\log (1/\varepsilon_n)}\approx \sqrt{\log n}$ and $K_n=\lfloor {3L A_n \over b_1\varepsilon_n}\rfloor+1$ with  $\varepsilon_n=\big({\sqrt{\log n}\over n}\big)^{1/3}$. Then $K_n\approx (n\log n)^{1/3}\approx n\varepsilon_n^2$. Write $\beta_0=(\beta_{0,1},\beta_{0,2},\dots,\beta_{0,K_n})$ for $\beta_{0,k}=f_0\big(-A_n+ {2A_nk-1\over K_n}\big)$. Since $f_0\in {\cal F}\cap Lip_L$, we have that
$f_{\beta_0}\in{\cal F}$ and $\sup_{-A_n\leq x< A_n}|\,f_{\beta_0}(x)-f_0(x)\,|\leq L A_n/ K_n\leq b_1 \varepsilon_n/3.$ From the triangle inequality and the inequality $\int_x^\infty \phi(t)dt\leq \phi(x)/x$ for all $x>0$, it follows that for all $f_\beta\in {\cal F}_n$ and for all large $n$,
$$\big|\, ||f_\beta-f_0||_2-||f_\beta-f_{\beta_0}||_2\,\big|\leq ||f_{\beta_0}-f_0||_2=\Big(\int_{-A_n}^{A_n}|f_0-f_{\beta_0}|^2\,d\nu\Big)^{1/2}$$$$+
\Big(\int_{I_0}f_0^2\,d\nu\Big)^{1/2}\leq {b_1 \varepsilon_n\over 3} \Big(\int_{-A_n}^{A_n}d\nu\Big)^{1/2}+M
\Big({\phi(A_n)\over A_n}\Big)^{1/2}$$$$\leq {b_1 \varepsilon_n\over 3}+{M\varepsilon_n\over (2\pi)^{1/4}A_n^{1/2}}\leq {b_1 \varepsilon_n\over 2}.$$
Thus for all large $j$ and n, we have
$${\Pi_n\big(f_\beta\in {\cal F}_n: j\varepsilon_n < ||f_\beta-f_0||_2\leq 2j\varepsilon_n\big)\over \Pi_n\bigl( W^1_n(\theta_0,\varepsilon_n) \bigr)}\leq {\Pi_n\big(f_\beta\in {\cal F}_n: ||f_\beta-f_0||_2\leq 2j\varepsilon_n\big)\over  \Pi_n\big(f_\beta\in {\cal F}_n: ||f_\beta-f_0||_2\leq b_1\varepsilon_n\big) }$$
$$\leq {\Pi_n\big(f_\beta\in {\cal F}_n: ||f_\beta-f_0||_2\leq 3j\varepsilon_n\big)\over  \Pi_n\big(f_\beta\in {\cal F}_n: ||f_\beta-f_{\beta_0}||_2\leq {b_1\over 2}\varepsilon_n\big) }$$$$={\Pi_n\big(\beta\in [-M,M]^{K_n}: ||\beta-\beta_0||_*\leq 3j\varepsilon_n\big)\over  \Pi_n\big(\beta\in [-M,M]^{K_n}:  ||\beta-\beta_0||_*\leq {b_1\over 2}\varepsilon_n\big) }.$$
Note that the Euclidean volume of the $K_n$-dimensional ellipsoid $\{\beta\in \R^{K_n}: ||\beta-\beta_0||_*\leq r\}$ is equal to $r^{K_n}$ times the Euclidean volume of the "unit" $K_n$-dimensional ellipsoid $\{\beta\in \R^{K_n}: ||\beta-\beta_0||_*\leq 1\}$. So the last quotient doer not exceed  $j^{2K_n}=e^{K_n\log (2j)}$, which is less than
$e^{c_2j^2n\varepsilon_n^2}$ for any given $c_2>0$ and all large $j$. Hence we have obtained condition (ii) of Corollary \ref{cor:4}. Similarly, for all large $j$ and n, we have
$$ N\big({\sqrt{a_0}\over 4\sqrt{a_1}}j\varepsilon_n,\{f_\beta\in {\cal F}_n: j\varepsilon_n < ||f_\beta-f_0||_2\leq 2j\varepsilon_n\},||\cdot||_2\big)$$
$$\leq  N\big({\sqrt{a_0}\over 4\sqrt{a_1}}j\varepsilon_n,\{f_\beta\in {\cal F}_n: ||f_\beta-f_{\beta_0}||_2\leq 3j\varepsilon_n\},||\cdot||_2\big)$$
$$\leq  N\big({\sqrt{a_0}\over 4\sqrt{a_1}}j\varepsilon_n,\{\beta\in [-M,M]^{K_n}: ||\beta-\beta_0||_*\leq 3j\varepsilon_n\},||\cdot||_*\big),$$
which, by Lemma 4.1 in Pollard \cite{pol}, is less than $ b_2^{K_n}= e^{K_n\log b_2}$ for some constant $b_2>0$, and therefore condition (i) of Corollary \ref{cor:4} holds for any given $c_1>0$.

\bigskip

\noindent{\it 4.2. Infinite-dimensional normal model.} We observe an infinite-dimensional random vector $(X_1,X_2,\dots)$, where the random vector $X^{(n)}=(X_1,\dots,X_n)$ for each $n$ is normally distributed according to $N(\theta_{(n)},\Sigma_{(n)})$ with density $p_{\theta_{(n)}}^{(n)}(x^{(n)})$, $\theta_{(n)}=(\theta_1,\dots,\theta_n)$, and the covariance matrix $\Sigma_{(n)}$ is known and satisfies $$\alpha\Sigma_{(n)}^{-1}\alpha^T\approx n\sum_{i=1}^n\alpha_i^2\eqno (a)$$
for all $\alpha=(\alpha_1,\dots,\alpha_n)\in \R^n$ and for all $n$. The parameter space $\Theta$ consists of all vectors $\theta=(\theta_1,\theta_2,\dots)$ in $\R^\infty$ with $||\theta||_2:=\big(\sum_{i=1}^\infty\theta_i^2\big)^{1/2}<\infty$. In this section we identify $\theta_{(n)}=(\theta_1,\dots,\theta_n)$ with $(\theta_1,\dots,\theta_n,0,0,\dots)$ and  hence the norm $||\theta_{(n)}||_2$ makes sense.
Let $\gamma$ be a positive constant. The true parameter $\theta_0=(\theta_{0,1},\theta_{0,2},\dots)$ is assumed to satisfy $$\sum_{i=1}^\infty\theta_{0,i}^2i^{2\gamma}<\infty.\eqno (b)$$
In the special case that $X_1,X_2,\dots$ are independent random variables and each $X_i$ is normally distributed with mean $\theta_i$ and variance $1/n$, the Bayesian estimation problem on parameters $\theta=(\theta_1,\theta_2,\dots)$ has been studied by many authors including Cox \cite{cox}, Freedman \cite{fr1}, Ghosal and van der Vaart \cite{ghv1}, Scricciolo \cite{scr}, Shen and Wasserman \cite{shw} and Zhao \cite{zha}. They showed that posteriors can attain the minimax rate $n^{-\gamma/(2\gamma+1)}$. Observe that every white noise model can be described as an infinite-dimensional normal model via an orthonormal basis.

Now we construct a prior such that the posterior attains the optimal rate of convergence in our framework. We put the prior on the parameter $\theta=(\theta_1,\theta_2,\dots)$ such that $\theta_{(k)}=(\theta_1,\dots,\theta_k)$ is distributed as $N(0,\Sigma_k)$ and that $\theta_{k+1},\theta_{k+2},\dots$ are set to be zero, where $k=\lfloor n^{1/(2\gamma+1)}c\rfloor$ with some positive constant $c$ which is determined later and the covariance matrix $\Sigma_k$ is assumed to satisfy $$\alpha\Sigma_k^{-1}\alpha^T\lesssim k\sum_{i=1}^k\alpha_i^2i^{2\gamma}\eqno (c)$$ for all $\alpha=(\alpha_1,\dots,\alpha_k)\in\R^k$ and for all such $k$. For instance, the last inequality holds if eigenvalues $\lambda_1\leq \lambda_2\leq \dots\leq\lambda_k$ of positive definite matrices $\Sigma_k^{-1}$ satisfy $\lambda_i\leq k \,i^{2\gamma}$ for $i=1,2,\dots,k$, which for independent variables $X_1,X_2,\dots$ is slightly weaker than the condition (7.8) given in Ghosal and van der Vaart \cite{ghv1}.
In the following we shall apply Corollary \ref{cor:0} to show that the corresponding posterior converges at the rate $\varepsilon_n=n^{-\gamma/(2\gamma+1)}$.

\begin{thm}\label{thm:9} Assume that $(a),\,(b)$ and $(c)$ hold. Let $k=\lfloor n^{1/(2\gamma+1)}c\rfloor$ and $\varepsilon_n=n^{-\gamma/(2\gamma+1)}$. Then there exist constants $c>0$ and  $r>0$ such that
$$\Pi_n\bigl(\theta\in \Theta:\,||\theta-\theta_0||_2\geq r\,\varepsilon_n|X^{(n)}\bigr)\longrightarrow 0$$
almost surely as $n\to\infty$.
\end{thm}

\begin{proof}For any $\alpha_1=(\theta_{1,1},\theta_{1,2},\dots,\theta_{1,n})$ and $\alpha_2=(\theta_{2,1},\theta_{2,2},\dots,\theta_{2,n})$ we have
$$H(p_{\alpha_1}^{(n)},p_{\alpha_2}^{(n)})^2=2-2\int_{\R^n} \sqrt{p_{\alpha_1}^{(n)}(x)p_{\alpha_2}^{(n)}(x)}\,dx=2-$$
$$-{2\over (2\pi)^{n/2}\sqrt{\det \Sigma_{(n)}}}\int_{\R^n} \exp\Big(-{1\over 4}\big((x-\alpha_1) \Sigma_{(n)}^{-1}(x-\alpha_1)^T +(x-\alpha_2) \Sigma_{(n)}^{-1}(x-\alpha_2)^T\big)\Big)dx,$$
where $x=(x_1,x_2,\dots,x_n)$ and
$$(x-\alpha_1) \Sigma_{(n)}^{-1}(x-\alpha_1)^T +(x-\alpha_2) \Sigma_{(n)}^{-1}(x-\alpha_2)^T$$
$$=2x\Sigma_{(n)}^{-1}x^T-2(\alpha_1+\alpha_2)\Sigma_{(n)}^{-1}x^T+\alpha_1\Sigma_{(n)}^{-1}\alpha_1^T+\alpha_2\Sigma_{(n)}^{-1}\alpha_2^T$$
$$=2(x-{\alpha_1\over 2}-{\alpha_2\over 2})\Sigma_{(n)}^{-1}(x-{\alpha_1\over 2}-{\alpha_2\over 2})^T-{1\over 2}(\alpha_1+\alpha_2)\Sigma_{(n)}^{-1}(\alpha_1+\alpha_2)^T+\alpha_1\Sigma_{(n)}^{-1}\alpha_1^T+\alpha_2\Sigma_{(n)}^{-1}\alpha_2^T$$
$$=2(x-{\alpha_1\over 2}-{\alpha_2\over 2})\Sigma_{(n)}^{-1}(x-{\alpha_1\over 2}-{\alpha_2\over 2})^T+{1\over 2}(\alpha_1-\alpha_2)\Sigma_{(n)}^{-1}(\alpha_1-\alpha_2)^T$$
$$\geq 2(x-{\alpha_1\over 2}-{\alpha_2\over 2})\Sigma_{(n)}^{-1}(x-{\alpha_1\over 2}-{\alpha_2\over 2})^T+b_1\,n\,||\alpha_1-\alpha_2||_2^2$$
for some positive constant $b_1$ independent of $\alpha_1,\,\alpha_2$, where the last inequality follows from condition $(a)$. Hence we get
$$H(p_{\alpha_1}^{(n)},p_{\alpha_2}^{(n)})^2\geq 2-2\,e^{-{b_1\over 4}\,n\,||\alpha_1-\alpha_2||_2^2},$$
which implies that the norm $2^{-1}b_1||\cdot||_2$ satisfies the inequality (1). So Corollary \ref{cor:0} can be applied for the metric $2^{-1}b_1||\cdot||_2$ and for constants $\alpha=1/2$ and $\delta=1/4$.

It follows from condition $(b)$ that $||\theta_{(k)}-\theta_0||_2^2=\sum_{i=1}^k(\theta_i-\theta_{0,i})^2+\sum_{i=k+1}^\infty\theta_{0,i}^2
\leq ||\theta_{(k)}-\theta_{0,(k)}||_2^2+k^{-2\gamma}\sum_{i=k+1}^\infty\theta_{0,i}^2i^{2\gamma}= ||\theta_{(k)}-\theta_{0,(k)}||_2^2+{\rm O}(\varepsilon_n^2),$ where $\theta_{(k)}=(\theta_1,\dots,\theta_k)$ and $\theta_{0,(k)}=(\theta_{0,1},\dots,\theta_{0,k}).$ This implies that for each large $j$,
$$C({1\over 4} j\varepsilon_n,\{\theta_{(k)}: j\varepsilon_n < {b_1\over 2}||\theta_{(k)}-\theta_0||_2 \leq 2j\varepsilon_n\},{1\over 2},{b_1\over 2}||\cdot||_2)$$
$$\leq C({1\over 5} j\varepsilon_n,\{\theta_{(k)}: ||\theta_{(k)}-\theta_{0,(k)}||_2 \leq 3j\varepsilon_n\},{1\over 2},||\cdot||_2),$$
which by Lemma 1 in Xing and Ranneby \cite{xir1} does not exceed
$$\Pi_n(\theta_{(k)}: ||\theta_{(k)}-\theta_{0,(k)}||_2 \leq 3j\varepsilon_n)^{1\over 2}
N({1\over 5} j\varepsilon_n,\{\theta_{(k)}: ||\theta_{(k)}-\theta_{0,(k)}||_2 \leq 3j\varepsilon_n\},||\cdot||_2)^{1\over 2}$$
$$\leq \Pi_n(\theta_{(k)}: ||\theta_{(k)}-\theta_{0,(k)}||_2 \leq 3j\varepsilon_n)^{1\over 2}\,b_2^k
$$
$$\leq \Pi_n(\theta_{(k)}: ||\theta_{(k)}-\theta_{0,(k)}||_2 \leq 3j\varepsilon_n)^{1\over 2}\,e^{{1\over 40}j^2n\varepsilon_n^2}
$$
for some constant $b_2>1$ and all large $j$, $n$, where we have applied Lemma 4.1 in Pollard \cite{pol}. It remains to prove that for large $j$ and $n$,
$$\Pi_n(\theta_{(k)}: ||\theta_{(k)}-\theta_{0,(k)}||_2 \leq 3j\varepsilon_n)\leq e^{{1\over 20}j^2n\varepsilon_n^2}\, \Pi_n\bigl( W_n(\theta_0,\varepsilon_n) \bigr).$$
By the proof of Lemma 1 in Xing \cite{xi1} we have
$$1+{3\over 2}\,H_*(p_{\theta_{0,(n)}}^{(n)},p_{\theta_{(n)}}^{(n)})^2=E_{\theta_{0,(n)}}\sqrt{p_{\theta_{0,(n)}}^{(n)}\Big/ p_{\theta_{(n)}}^{(n)}}={1\over (2\pi)^{n/2}\sqrt{\det \Sigma_{(n)}}}$$
$$\int_{\R^n} \exp\Big(-{3\over 4}(x-\theta_{0,(n)}) \Sigma_{(n)}^{-1}(x-\theta_{0,(n)})^T +{1\over 4}(x-\theta_{(n)}) \Sigma_{(n)}^{-1}(x-\theta_{(n)})^T\Big)dx.$$
Write
$$-{3\over 4}(x-\theta_{0,(n)}) \Sigma_{(n)}^{-1}(x-\theta_{0,(n)})^T +{1\over 4}(x-\theta_{(n)}) \Sigma_{(n)}^{-1}(x-\theta_{(n)})^T$$
$$=-{1\over 2}(x-\theta_{0,(n)}) \Sigma_{(n)}^{-1}(x-\theta_{0,(n)})^T+{1\over 2}(\theta_{0,(n)}-\theta_{(n)}) \Sigma_{(n)}^{-1}x^T  $$$$-{1\over 4}\theta_{0,(n)} \Sigma_{(n)}^{-1}\theta_{0,(n)}^T+{1\over 4}\theta_{(n)} \Sigma_{(n)}^{-1}\theta_{(n)}^T$$
$$=-{1\over 2}(x-\theta_{0,(n)}) \Sigma_{(n)}^{-1}(x-\theta_{0,(n)})^T+{1\over 2}(\theta_{0,(n)}-\theta_{(n)}) \Sigma_{(n)}^{-1}(x-\theta_{0,(n)})^T  $$$$+{1\over 4}(\theta_{0,(n)}-\theta_{(n)}) \Sigma_{(n)}^{-1}(\theta_{0,(n)}-\theta_{(n)})^T$$
$$=-{1\over 2}\big((x-\theta_{0,(n)}) \Sigma_{(n)}^{-1}(x-\theta_{0,(n)})^T-(\theta_{0,(n)}-\theta_{(n)}) \Sigma_{(n)}^{-1}(x-\theta_{0,(n)})^T  $$$$+{1\over 4}(\theta_{0,(n)}-\theta_{(n)}) \Sigma_{(n)}^{-1}(\theta_{0,(n)}-\theta_{(n)})^T\big)+{3\over 8}(\theta_{0,(n)}-\theta_{(n)}) \Sigma_{(n)}^{-1}(\theta_{0,(n)}-\theta_{(n)})^T$$
$$=-{1\over 2}(x-{3\over 2}\theta_{0,(n)}+{1\over 2}\theta_{(n)}) \Sigma_{(n)}^{-1}(x-{3\over 2}\theta_{0,(n)}+{1\over 2}\theta_{(n)})^T+{3\over 8}(\theta_{0,(n)}-\theta_{(n)}) \Sigma_{(n)}^{-1}(\theta_{0,(n)}-\theta_{(n)})^T.$$
Hence we obtain
$$1+{3\over 2}\,H_*(p_{\theta_{0,(n)}}^{(n)},p_{\theta_{(n)}}^{(n)})^2=\exp\Big({3\over 8}(\theta_{0,(n)}-\theta_{(n)}) \Sigma_{(n)}^{-1}(\theta_{0,(n)}-\theta_{(n)})^T\Big).$$
It then follows from condition $(a)$ that there exists a positive constant $b_3$ not depending on $n$ such that
$$1+{3\over 2}\,H_*(p_{\theta_{0,(n)}}^{(n)},p_{\theta_{(n)}}^{(n)})^2\leq e^{{3\over 2}b_3\,\,n\,||\theta_{0,(n)}-\theta_{(n)}||_2^2}.$$
The constant $c$ is now chosen so largely that
$b_3||\theta_{(k)}-\theta_{0,(n)}||_2^2\leq b_3||\theta_{(k)}-\theta_{0,(k)}||_2^2+2^{-1}\varepsilon_n^2.$ Since the support $\Pi_n$ is $\{(\theta_1,\theta_2,\dots):\theta_l=0\ {\rm for}\ l\geq k+1\}$, we get
$$\Pi_n\bigl( W_n(\theta_0,\varepsilon_n)\bigr)\geq
\Pi_n(\theta_{(k)}: ||\theta_{(k)}-\theta_{0,(k)}||_2\leq (2b_3)^{-1/2}\varepsilon_n)$$
and hence
$${\Pi_n(\theta_{(k)}: ||\theta_{(k)}-\theta_{0,(k)}||_2 \leq 3j\varepsilon_n)\over \Pi_n\bigl( W_n(\theta_0,\varepsilon_n) \bigr)} \leq {\Pi_n(\theta_{(k)}: ||\theta_{(k)}-\theta_{0,(k)}||_2 \leq 3j\varepsilon_n)\over \Pi_n(\theta_{(k)}: ||\theta_{(k)}-\theta_{0,(k)}||_2 \leq (2b_3)^{-1/2}\varepsilon_n)}$$
$$={\int_{ ||\theta_{(k)}-\theta_{0,(k)}||_2 \leq 3j\varepsilon_n} \exp\big(-{1\over 2}\theta_{(k)} \Sigma_{k}^{-1}\theta_{(k)}^T \big)d\theta_{(k)}\over \int_{ ||\theta_{(k)}-\theta_{0,(k)}||_2 \leq (2b_3)^{-1/2}\varepsilon_n} \exp\big(-{1\over 2}\theta_{(k)} \Sigma_{k}^{-1}\theta_{(k)}^T \big)d\theta_{(k)}}$$
$$\leq{\int_{ ||\theta_{(k)}-\theta_{0,(k)}||_2 \leq 3j\varepsilon_n} d\theta_{(k)}\over \min\limits_{ ||\theta_{(k)}-\theta_{0,(k)}||_2 \leq (2b_3)^{-1/2}\varepsilon_n}
\exp\big(-{1\over 2}\theta_{(k)} \Sigma_{k}^{-1}\theta_{(k)}^T \big)\int_{ ||\theta_{(k)}-\theta_{0,(k)}||_2 \leq (2b_3)^{-1/2}\varepsilon_n}d\theta_{(k)}}$$
$$=\max\limits_{ ||\theta_{(k)}-\theta_{0,(k)}||_2 \leq (2b_3)^{-1/2}\varepsilon_n}
\exp\big({1\over 2}\theta_{(k)} \Sigma_{k}^{-1}\theta_{(k)}^T \big){ (3j)^k\over \big( (2b_3)^{-1/2}\big)^k}$$
$$\leq e^{{1\over 40}j^2n\varepsilon_n^2}\max\limits_{ ||\theta_{(k)}-\theta_{0,(k)}||_2 \leq (2b_3)^{-1/2}\varepsilon_n}
\exp\big({1\over 2}\theta_{(k)} \Sigma_{k}^{-1}\theta_{(k)}^T \big)$$
for all large $j$ and $n$. On the other hand, it turns out from condition (c) that there exists $b_4>0$ such that
for any $\theta_{(k)}=(\theta_1,\dots,\theta_k)$ with $||\theta_{(k)}-\theta_{0,(k)}||_2 \leq (2b_3)^{-1/2}\varepsilon_n$, we have
$$\exp\big({1\over 2}\theta_{(k)} \Sigma_{k}^{-1}\theta_{(k)}^T \big)\leq \exp\big(b_4k\sum_{i=1}^k\theta_i^2i^{2\gamma} \big)$$$$\leq \exp\big(2b_4k\sum_{i=1}^k(\theta_i-\theta_{0,i})^2i^{2\gamma}+2b_4k\sum_{i=1}^k\theta_{0,i}^2i^{2\gamma} \big)$$
$$\leq \exp\big(2b_4k^{2\gamma+1}\sum_{i=1}^k(\theta_i-\theta_{0,i})^2+2b_4k\sum_{i=1}^\infty\theta_{0,i}^2i^{2\gamma} \big)$$
$$\leq \exp\big(b_4b_3^{-1}k^{2\gamma+1}\varepsilon_n^2+2b_4k\sum_{i=1}^\infty\theta_{0,i}^2i^{2\gamma} \big)\leq e^{{1\over 40}j^2n\varepsilon_n^2}$$
for all large $j$ and $n$, where the second inequality follows from the inequality $(s+t)^2\leq 2s^2+2t^2$ for all $s,t\in\R$. Therefore, we have proved the required inequality and the proof of Theorem \ref{thm:9} is complete.

\end{proof}

\bigskip

\noindent{\it 4.3. Prior based on uniform distributions.} Assume, just as in Section 3.1, that
$(X_1,X_2,\dots,X_n)$ of independent variables $X_i$ has a density $\prod_{i=1}^np_{\theta,i}(x_i)$ relative to the product measure $\mu_1\times\mu_2\times\dots\times\mu_n$ on $\mathfrak{X}_1\times \mathfrak{X}_2\times\dots\times \mathfrak{X}_n$. We follow the notations of Section 3.1.
By means of the componentwise Hellinger upper bracketing numbers for $\Theta$, Ghosal and van der Vaart \cite{ghv1} have obtained an in-probability convergence rate theorem for priors based on discrete distributions. Their result can be extended to an almost sure assertion in terms of Theorem \ref{thm:4}. In the following we give an almost sure result for priors based on uniform distributions, which gives us an opportunity to adopt the average Hellinger metric $d_n^0(\theta_1,\theta_2)=\big({1\over n}\sum_{i=1}^nH_i(p_{\theta_1,i},p_{\theta_2,i})^2\big)^{1/2}$ instead of the componentwise Hellinger upper bracketing numbers. This also extends a result for i.i.d. observations given by Xing (\cite{xi1}, Section 3.2).

Let $c>1$ and let $\bar{d}_n$ be metrics on $\Theta$. Assume that $\Theta_{c,n}$ for $n=1,2\dots$ are subsets of $ \Theta$ such that ${1\over c^2\, n}\sum_{i=1}^nH_{*,i}(p_{\theta_1,i},p_{\theta_2,i})^2\leq \bar{d}_n(\theta_1,\theta_2)^2$
for all $\theta_1,\theta_2\in \Theta_{c,n}$. By the definition of $H_{*,i}$ we have $d_n^0\leq \sqrt{3}c\,\bar{d}_n$ on $\Theta_{c,n}$.
Note that $\bar{d}_n$ can be taken as a constant multiple of $d_n^0$ in the case that $H_{*,i}(p_{\theta_1,i},p_{\theta_2,i})\lesssim H_i(p_{\theta_1,i},p_{\theta_2,i})$ for all $\theta_1,\theta_2$ in $\Theta$ and $i=1,2,\dots,n$.
Given  $\varepsilon_n>0$, we assume that $\{B_1,\dots,B_{K_n}\}$ is a partition of $\Theta_{c,n}$ such that for each $B_i$ there exists $b_i$ in $\Theta$ with $B_i\subset \{\theta\in \Theta_{c,n}:\, \bar{d}_n(b_i,\theta)\leq \varepsilon_n/2\,c\}$. Let $\Pi_n$ be a prior distribution supported on $\Theta_{c,n}$ such that $\Pi_n(B_i)=1/K_n$ for $i=1,2,\dots,K_n$. Corollary \ref{cor:3} implies the following result.

\begin{thm}\label{thm:8} Suppose that $\theta_0\in \Theta_{c,n}$ for all $n$ and suppose that $\log K_n+\log n=O(n\,\varepsilon_n^2)$ as $n\to\infty$. Then for each large $r$,
$$\Pi_n\bigl(\theta\in \Theta:\,d_n^0(\theta,\theta_0)\geq r\,\varepsilon_n|X_1,X_2,\dots,X_n\bigr)\longrightarrow 0 $$
almost\ surely as $n\to\infty$.
\end{thm}

\begin{proof} Take $\Theta_n=\Theta_{c,n}$ for all $n$. Then condition (iii) of Corollary \ref{cor:3} is trivially fulfilled.
For $\delta=1/(2\sqrt{3}c^2)$ we have that for any given $c_1>0$ and all large $j$ and $n$,
$$N\big(\delta j\varepsilon_n,\{\theta\in\Theta_n: j\varepsilon_n < d_n^0(\theta,\theta_0)\leq 2j\varepsilon_n\},d_n^0\big)\leq N\big( {\varepsilon_n\over 2\sqrt{3}c^2},\Theta_n,d_n^0\big)$$
$$\leq N\big( {\varepsilon_n\over 2c},\Theta_n,\bar{d}_n\big)\leq K_n \leq e^{c_1j^2n\varepsilon_n^2},$$  where the last inequality follows from $\log K_n=O(n\,\varepsilon_n^2)$. This implies condition (i) of Corollary \ref{cor:3}. To see condition (ii), by $\theta_0\in \Theta_{c,n}$ we can take $b_{i_0}\in \Theta$ such that $\bar{d}_n(b_{i_0},\theta_0)\leq \varepsilon_n/2c$.
Then, for all $\theta\in B_{i_0}$ we have  $${1\over n}\sum_{i=1}^nH_{*,i}(p_{\theta_0,i},p_{\theta,i})^2\leq c^2\bar{d}_n(\theta_0,\theta)^2\leq c^2\big(\bar{d}_n(\theta_0,b_{i_0})+\bar{d}_n(b_{i_0},\theta)\big)^2\leq \varepsilon_n^2,$$ which implies that $\overline{W}_n(\theta_0,\varepsilon_n)$ contains the whole set $B_{i_0}$ and hence
$\Pi_n\bigl( \overline{W}_n(\theta_0,\varepsilon_n) \bigr)\geq \Pi_n( B_{i_0})=1/K_n \geq e^{-c_2j^2n\varepsilon_n^2}$ for any given $c_2>0$ and all large $j$ and $n$. So we have verified condition (ii) and the proof of Theorem \ref{thm:8} is complete.
\end{proof}

\noindent{\it Example} (Nonparametric Poisson regression) Assume that $U\geq L>0$ are two given constants. We consider
Poisson distributed independent random variables $X_1,X_2,\dots,X_n$ with parameters $\theta(z_1),\ \theta(z_2),\ \dots,\theta(z_n)$, where $\theta:\R\to [L,U]$ is an unknown increasing link function and $z_1,z_2,\dots,z_n$ are one-dimensional covariates. The joint mass function of $(X_1,X_2,\dots,X_n)$ is given by $\prod\limits_{i=1}^n p_{\theta,i}(x_i)$ with $ p_{\theta,i}(x_i)=e^{-\theta(z_i)}{\theta(z_i)^{x_i}\over x_i!}.$
For $a,b\in  [L,U]$ we have
$$\sum_{x=0}^\infty \Big(\ \sqrt{e^{-a}{a^x\over x!}}-\sqrt{e^{-b}{b^x\over x!}}\ \Big)^2\Big({2\over 3}\sqrt{e^{-a}{a^x\over x!}\over e^{-b}{b^x\over x!}}+{1\over 3}\Big)$$
$$= \sum_{x=0}^\infty {( e^{-{a\over 2}}a^{x\over 2}-e^{-{b\over 2}}b^{x\over 2} )^2 \over x!} \Big({2\over 3}\sqrt{e^{b-a}a^x\over b^x}+{1\over 3}\Big)$$
$$\leq (a-b)^2e^{-L}\sum_{x=0}^\infty {(U^{x\over 2}+xU^{{x\over 2}-1})^2\over x!} \Big({2\over 3}\sqrt{e^{U-L}U^x\over L^x}+{1\over 3}\Big)$$
$$\leq (a-b)^2e^{U-3L\over 2}\sum_{x=0}^\infty {(U^{x\over 2}+xU^{{x\over 2}-1})^2\over x!} \Big({U\over L}\Big)^{x\over 2}\lesssim (a-b)^2,$$
where the first inequality follows from the inequality
$| e^{-{a\over 2}}a^{x\over 2}-e^{-{b\over 2}}b^{x\over 2}|\leq |a-b|e^{-{L\over 2}}(U^{x\over 2}+xU^{{x\over 2}-1})$ for all $a,b\in [L,U]$. This implies that ${1\over  n}\sum_{i=1}^nH_{*,i}(p_{\theta_1,i},p_{\theta_2,i})^2\lesssim \int(\theta_1-\theta_2)^2d\mathbb{P}_n^z$ for all link functions $\theta_1$ and $\theta_2$, where $\mathbb{P}_n^z=n^{-1}\sum_{i=1}^n\delta_{z_i}$ denotes the empirical distribution of $z_1,z_2,\dots,z_n$. So one can use the $L_2(\mathbb{P}_n^z)$-matric to produce the  partition $\{B_1,\dots,B_{K_n}\}$ of the space of link functions. By Theorem 2.7.5 of \cite{vaw} we know that $\log K_n\lesssim \varepsilon_n^{-1}$. Letting $\varepsilon_n^{-1}=n\varepsilon_n^2$ we obtain $\varepsilon_n=n^{-1/3}$, and hence by Theorem \ref{thm:8} the posterior based on uniform distributions converges almost surely at the rate $\varepsilon_n=n^{-1/3}$ with respect to the metric $d_n^0$, which is the minimax rate for this model. The in-probability convergence rate $n^{-1/3}$ for the posterior based on discrete distributions has been obtained in Section 7.1.1 of Ghosal and van der Vaart \cite{ghv1}.

It is worth pointing out that in this example the suprenorm $||p_{\theta_1,i}/p_{\theta_2,i}||_\infty$ may not be finite. Therefore, the approach on determination of prior concentration rates by means of $H(p_{\theta_1,i},p_{\theta_2,i})\,||p_{\theta_1,i}/p_{\theta_2,i}||_\infty$ in Ghosal, Ghosh and van der Vaart \cite{ghg} fails to be applied in this case, but the modified Hellinger distance $H_*(p_{\theta_1,i},p_{\theta_2,i})$ works well. A similar argument holds even for the infinite-dimensional normal model.

\section{Appendix}
\begin{proof}[Proof of Proposition \ref{prop:1}.]
Given $\delta>1$, by the definition of the Hausdorff $\alpha$-constant and Assumption \ref{ass:1}, there exist pairwise disjoint subsets $B_1,B_2,\dots,B_{N_n}$ of $\Theta_1$ such that (1) $\cup_{k=1}^{N_n}B_k=\{\theta\in \Theta_1:\, d_n(\theta,\theta_0)>\varepsilon\}$; (2) each $B_k$ is contained in some ball of $e_n$-radius not exceeding $\varepsilon$; (3) $\sum_{k=1}^{N_n}\Pi_n(B_k)^\alpha\leq \delta C(\varepsilon,\{\theta\in \Theta_1:\, d_n(\theta,\theta_0)>\varepsilon\},\alpha,e_n);$ (4) there exist test functions $\phi_k$ such that
$P_{\theta_0}^{(n)}\phi_k\leq e^{-Kn\varepsilon^2}$ and $ P_{\theta}^{(n)}\phi_k\geq 1-e^{-Kn\varepsilon^2}$ for all $\theta$ in $B_k$.
Then by the inequality $(x+y)^\alpha\leq x^\alpha+y^\alpha$ for all $x,y\geq 0$, we get
$$P_{\theta_0}^{(n)}\Big(\int_{\theta\in \Theta_1:\,d_n(\theta,\theta_0)> \varepsilon}R_\theta^{(n)}(X^{(n)}) \Pi_n(d\theta)\Big)^\alpha\leq  \sum_{k=1}^{N_n}P_{\theta_0}^{(n)}\Big( \int_{B_k}R_\theta^{(n)}(X^{(n)}) \Pi_n(d\theta)\Big)^\alpha$$
$$ \leq \sum_{k=1}^{N_n}
P_{\theta_0}^{(n)}\Bigl\{\phi_k^{1-\alpha}\Big(\int_{B_k}R_\theta^{(n)}(X^{(n)})\, \Pi_n(d\theta)\Big)^\alpha\Bigr\}
$$
$$+ \sum_{k=1}^{N_n}P_{\theta_0}^{(n)}\Bigl\{(1-\phi_k)^{1-\alpha}\Big(\int_{B_k}R_\theta^{(n)}(X^{(n)}) \Pi_n(d\theta)\Big)^\alpha\Bigr\} :=L_1+L_2.$$
It turns out from H\"older's inequality and Fubini's theorem that
$$ L_1\leq \sum_{k=1}^{N_n} \bigl(P_{\theta_0}^{(n)}\phi_k\bigr)^{1-\alpha}\Big(P_{\theta_0}^{(n)}\int_{B_k}R_\theta^{(n)}(X^{(n)}) \Pi_n(d\theta)\Big)^\alpha$$
$$\leq e^{-(1-\alpha)Kn\varepsilon^2}\sum_{k=1}^{N_n}\Big(\int_{B_k}P_{\theta_0}^{(n)}R_\theta^{(n)}(X^{(n)}) \Pi_n(d\theta)\Big)^\alpha= e^{-(1-\alpha)Kn\varepsilon^2}\sum_{k=1}^{N_n} \Pi_n(B_k)^\alpha $$
$$\leq \delta e^{-(1-\alpha)Kn\varepsilon^2}C(\varepsilon,\{\theta\in \Theta_1:\, d_n(\theta,\theta_0)>\varepsilon\},\alpha,e_n).$$
To estimate $L_2$, we deal with $1/2\leq \alpha<1$ and $0<\alpha<1/2$ separately. In the case of $1/2\leq \alpha<1$ we have  $0\leq (2\alpha-1)/\alpha<1$ and by H\"older's inequality,
$$L_2= \sum_{k=1}^{N_n}P_{\theta_0}^{(n)}\Bigl\{(1-\phi_k)^{1-\alpha}\Big(\int_{B_k}R_\theta^{(n)}(X^{(n)}) \Pi_n(d\theta)\Big)^{1-\alpha}$$
$$\cdot\Big(\int_{B_k}R_\theta^{(n)}(X^{(n)}) \Pi_n(d\theta)\Big)^{2\alpha-1}\Bigr\}$$
$$\leq \sum_{k=1}^{N_n}\Bigl\{P_{\theta_0}^{(n)}\Big((1-\phi_k)\int_{B_k}R_\theta^{(n)}(X^{(n)}) \Pi_n(d\theta)\Big)\Bigr\}^{1-\alpha}$$
$$\cdot \Bigl\{P_{\theta_0}^{(n)}\Big(\int_{B_k}R_\theta^{(n)}(X^{(n)}) \Pi_n(d\theta)\Big)^{2\alpha-1\over \alpha}\Bigr\}^\alpha$$
$$\leq  \sum_{k=1}^{N_n}\Bigl(\int_{B_k}P_{\theta}^{(n)}(1-\phi_k)\, \Pi_n(d\theta) \Bigr)^{1-\alpha}\Bigl(P_{\theta_0}^{(n)}\int_{B_k}R_\theta^{(n)}(X^{(n)}) \Pi_n(d\theta)\Bigr)^{2\alpha-1}$$
$$\leq \sum_{k=1}^{N_n}
e^{-(1-\alpha)Kn\varepsilon^2}\Pi_n(B_k)^{1-\alpha}\Pi_n(B_k)^{2\alpha-1}= e^{-(1-\alpha)Kn\varepsilon^2}\sum_{k=1}^{N_n}
\Pi_n(B_k)^\alpha$$
$$\leq \delta e^{-(1-\alpha)Kn\varepsilon^2}C(\varepsilon,\{\theta\in \Theta_1:\, d_n(\theta,\theta_0)>\varepsilon\},\alpha,e_n).$$
In the case of $0<\alpha< 1/2$ we have $0\leq (1-\phi_k)^{1-\alpha}\leq (1-\phi_k)^\alpha\leq 1$ and hence by H\"older's inequality,
$$L_2\leq \sum_{k=1}^{N_n}P_{\theta_0}^{(n)}\Bigl\{\Big((1-\phi_k)\int_{B_k}R_\theta^{(n)}(X^{(n)}) \Pi_n(d\theta)\Big)^\alpha\Bigr\}$$
$$\leq  \sum_{k=1}^{N_n}\Bigl\{P_{\theta_0}^{(n)}\Big((1-\phi_k)\int_{B_k}R_\theta^{(n)}(X^{(n)}) \Pi_n(d\theta)\Big)\Bigr\}^\alpha$$
$$= \sum_{k=1}^{N_n}\Bigl(\int_{B_k}P_{\theta}^{(n)}(1-\phi_k) \Pi_n(d\theta)\Big)^\alpha\leq e^{-\alpha Kn\varepsilon^2}\sum_{k=1}^{N_n}
\Pi_n(B_k)^\alpha$$
$$\leq \delta e^{-\alpha Kn\varepsilon^2}C(\varepsilon,\{\theta\in \Theta_1:\, d_n(\theta,\theta_0)>\varepsilon\},\alpha,e_n).$$
Thus for any $0<\alpha<1$ we have obtained the required inequality for $K_1=2\delta$ and $K_2=\alpha\,K$ if $0<\alpha<1/2$ and $K_2=(1-\alpha)\,K$ if $1/2\leq \alpha<1$. Finally, letting $\delta\searrow 1$, we conclude the proof of Proposition \ref{prop:1}.
\end{proof}

\begin{proof}[Proof of Proposition \ref{prop:2}.] Take nonempty disjoint subsets $B_j,\ j=1,2,\dots,N$, of $\Theta$ such that $\sum_{j=1}^N \Pi(B_j)^\alpha\leq 2\,C(\delta\,\varepsilon,\{\theta\in \Theta_1:\,d_n^1(\theta,\theta_0)> \varepsilon\},\alpha,d_n^1)$, $\cup_{j=1}^N B_j= \{\theta\in \Theta_1:\,d_n^1(\theta,\theta_0)> \varepsilon\}$ and $d_n^1$-diameters of all $B_j$ do not exceed  $2\,\delta\,\varepsilon$. Then we have
$$P_{\theta_0}^{(n)}\,\Big(\int_{\theta\in \Theta_1:\,d_n^1(\theta,\theta_0)> \varepsilon}R_\theta^{(n)}(X^{(n)})\, \Pi_n(d\theta)\Big)^\alpha$$
$$\leq P_{\theta_0}^{(n)}\,\sum_{j=1}^N \Big(\int_{B_j}R_\theta^{(n)}(X^{(n)})\, \Pi_n(d\theta)\Big)^\alpha=\sum_{j=1}^N \Pi_n(B_j)^\alpha\, P_{\theta_0}^{(n)}\Big({I_j(X^{(n)})\over p_{\theta_0}^{(n)}(X^{(n)})}\Big)^\alpha$$
$$\leq 2\,C(\delta\,\varepsilon,\{\theta\in \Theta_1:\,d_n^1(\theta,\theta_0)> \varepsilon\},\alpha,d_n^1)\,
\max_{1\leq j\leq N}\, P_{\theta_0}^{(n)}\Big({I_j(X^{(n)})\over p_{\theta_0}^{(n)}(X^{(n)}) }\Big)^\alpha,$$
where $I_j(X^{(n)})=\Pi_n(B_j)^{-1}\int_{B_j}p_{\theta}^{(n)}(X^{(n)})\, \Pi_n(d\theta)$ is the integral mean of the likelihood $p_{\theta}^{(n)}(X^{(n)})$ and hence is a density function. With a slight abuse of notation we also let $I_j$ stand for the corresponding parameter of this integral means.
Take $\theta_j\in B_j$ for each $j$.
By Jensen's inequality for $d_n^1(\cdot,\theta_j)^s$ we have  $d_n^1(I_j,\theta_j)\leq 2\delta\varepsilon$ and thus
$d_n^1(I_j,\theta_0)\geq d_n^1(\theta_j,\theta_0)-d_n^1(I_j,\theta_j)\geq (1-2\delta)\,\varepsilon$.
Take an nonnegative integer $m$ with ${\alpha\over 1-\alpha}\leq 2^m<{2\alpha\over 1-\alpha}.$
From H\"older's inequality it turns out that for each $j$,
$$P_{\theta_0}^{(n)}\Big({I_j(X^{(n)})\over p_{\theta_0}^{(n)}(X^{(n)})}\Big)^\alpha=P_{\theta_0}^{(n)}\biggl(\Big({I_j(X^{(n)})\over p_{\theta_0}^{(n)}(X^{(n)})}\Big)^{\alpha\over 2}\,\Big({I_j(X^{(n)})\over p_{\theta_0}^{(n)}(X^{(n)})}\Big)^{\alpha\over 2}\biggr)$$
$$\leq \biggl(P_{\theta_0}^{(n)}\Big({I_j(X^{(n)})\over p_{\theta_0}^{(n)}(X^{(n)})}\Big)^{{\alpha\over 2}\cdot{2\over 2-\alpha}}\biggr)^{2-\alpha\over 2}\,\biggl(P_{\theta_0}^{(n)}\Big({I_j(X^{(n)})\over p_{\theta_0}^{(n)}(X^{(n)})}\Big)^{{\alpha\over 2}\cdot{2\over \alpha}}\biggr)^{\alpha\over 2}$$
$$=\biggl(P_{\theta_0}^{(n)}\Big({I_j(X^{(n)})\over p_{\theta_0}^{(n)}(X^{(n)})}\Big)^{\alpha\over 2-\alpha}\biggr)^{2-\alpha\over 2},$$
which, by repeating the above procedure $m-1$ more times, does not exceed
$$\biggl(P_{\theta_0}^{(n)}\Big({I_j(X^{(n)})\over p_{\theta_0}^{(n)}(X^{(n)})}\Big)^{\alpha\over 2^m(1-\alpha)+\alpha}\biggr)^{2^m(1-\alpha)+\alpha\over 2^m}   \leq\biggl(P_{\theta_0}^{(n)}\Big({I_j(X^{(n)})\over p_{\theta_0}^{(n)}(X^{(n)})}\Big)^{1\over 2}\biggr)^{\alpha\over 2^{m-1}}$$
$$=\Bigl(\ \int\sqrt{I_j(X^{(n)})\,p_{\theta_0}^{(n)}(X^{(n)})}\, \mu(dX^{(n)}) \Bigr)^{\alpha\over 2^{m-1}}$$
$$
= \Bigl(1-{H\big(I_j(X^{(n)})\,p_{\theta_0}^{(n)}(X^{(n)})\big)^2\over 2}\Bigr)^{\alpha\over 2^{m-1}}\leq  e^{-2^{-m}\,\alpha\,n\,d_n^1(I_j\,\theta_0)^2 }
$$$$\leq e^{-2^{-m}\,(1-2\delta)^2\,\alpha\,n\,\varepsilon^2}\leq e^{-{1\over 2}(1-\alpha)\,(1-2\delta)^2\,n\,\varepsilon^2},$$
which completes the proof of Proposition \ref{prop:2}.
\end{proof}

\begin{proof}[Proof of Proposition \ref{prop:0}.] Denote $S=\{\theta\in \Theta_1:\,d_n^1(\theta,\theta_0)> \varepsilon\}$. Assume first $0<\beta\leq 1/2$. By H\"older's inequality and the inequality $1-x\leq e^{-x}$, we have
$$P_{\theta_0}^{(n)}\Big(\int_SR_\theta^{(n)}(X^{(n)})^\beta\, \Pi_n(d\theta)\Big)^\alpha$$
$$\leq P_{\theta_0}^{(n)}\Big(\int_SR_\theta^{(n)}(X^{(n)})^{\beta\cdot {1\over 2\beta}}\, \Pi_n(d\theta)\Big)^{ 2\beta\alpha} \,\Pi_n (S)^{(1-2\beta)\alpha} $$
$$\leq \Big(P_{\theta_0}^{(n)}\int_SR_\theta^{(n)}(X^{(n)})^{1\over 2}\, \Pi_n(d\theta)\Big)^{ 2\beta\alpha} \,\Pi_n( S)^{(1-2\beta)\alpha} $$
$$= \Big(\int_SP_{\theta_0}^{(n)}R_\theta^{(n)}(X^{(n)})^{1\over 2}\, \Pi_n(d\theta)\Big)^{ 2\beta\alpha} \,\Pi_n( S)^{(1-2\beta)\alpha} $$
$$\leq \Big(\int_S e^{-{1\over 2}n\varepsilon^2}\ \Pi_n(d\theta)\Big)^{ 2\beta\alpha} \,\Pi_n( S)^{(1-2\beta)\alpha}= e^{-\beta\alpha n\varepsilon^2}\,\Pi_n( S)^{\alpha},$$
which gives the required inequality when $0<\beta\leq 1/2$. If $1/2<\beta<1$ we take $p={1\over 2-2\beta}$ and $q={1\over 2\beta-1}$. It then follows from H\"older's inequality that
$$P_{\theta_0}^{(n)}\Big(\int_SR_\theta^{(n)}(X^{(n)})^\beta\, \Pi_n(d\theta)\Big)^\alpha $$
$$\leq P_{\theta_0}^{(n)}\Biggl(\bigg(\int_SR_\theta^{(n)}(X^{(n)})^{{1\over 2p}\cdot p}\, \Pi_n(d\theta)\bigg)^{\alpha\over p}\, \bigg(\int_SR_\theta^{(n)}(X^{(n)})^{(\beta-{1\over 2p})\cdot q}\, \Pi_n(d\theta)\bigg)^{\alpha\over q}\Biggr)  $$
$$\leq \Biggl(P_{\theta_0}^{(n)}\bigg(\int_SR_\theta^{(n)}(X^{(n)})^{1\over 2}\, \Pi_n(d\theta)\bigg)^{\alpha}\Biggr)^{1\over p}\, \Biggl(P_{\theta_0}^{(n)}\bigg(\int_SR_\theta^{(n)}(X^{(n)})\, \Pi_n(d\theta)\bigg)^{\alpha}\Biggr)^{1\over q}  $$
$$\leq \bigg(P_{\theta_0}^{(n)}\int_SR_\theta^{(n)}(X^{(n)})^{1\over 2}\, \Pi_n(d\theta)\bigg)^{\alpha\over p}\,\bigg( P_{\theta_0}^{(n)}\int_SR_\theta^{(n)}(X^{(n)})\, \Pi_n(d\theta)\bigg)^{\alpha\over q}  $$
$$\leq \Big(\int_S e^{-{1\over 2}n\varepsilon^2}\ \Pi_n(d\theta)\Big)^{ \alpha\over p} \,\Pi_n( S)^{\alpha\over q}= e^{-(1-\beta)\alpha n\varepsilon^2}\,\Pi_n( S)^{\alpha}.$$
The proof of Proposition \ref{prop:0} is complete.
\end{proof}

To prove Theorem \ref{thm:1} we need two simple lemmas.

\begin{lem}\label{lem:2}Let $\varepsilon>0$ and $c>0$. Then the inequality
$$P_{\theta_0}^{(n)}\Bigl(\ \int_{\Theta} R_\theta^{(n)}(X^{(n)})\, \Pi_n(d\theta)\leq e^{-n\,\varepsilon^2\,(3+2c)}\ \Pi_n\bigl( W_n(\theta_0,\varepsilon) \bigr)\, \Bigr)\leq  e^{-n\,\varepsilon^2\, c} $$
holds for all $n$.
\end{lem}
\begin{proof} Without loss of generality, we may assume that $\Pi_n\bigl( W_n(\theta_0,\varepsilon)\bigr)>0$. From Jensen's inequality and Chebyshev's inequality it follows that
$$P_{\theta_0}^{(n)}\Bigl(\ \int_{\Theta} R_\theta^{(n)}(X^{(n)})\, \Pi_n(d\theta)\leq e^{-n\,\varepsilon^2\,(3+2c)}\ \Pi_n\bigl( W_n(\theta_0,\varepsilon) \bigr)\, \Bigr)$$
$$\leq P_{\theta_0}^{(n)}\biggl(\ e^{n\,\varepsilon^2\,({3\over 2}+c)}\leq  \Bigl(\,{1\over \Pi_n\bigl(W_n(\theta_0,\varepsilon)\bigr)}\,\int_{W_n(\theta_0,\varepsilon)}  R_\theta^{(n)}(X^{(n)})\, \Pi_n(d\theta)\Bigr)^{-{1\over 2}}\ \biggr)$$
$$\leq P_{\theta_0}^{(n)}\biggl(\ e^{n\,\varepsilon^2\,({3\over 2}+c)}\leq  {1\over \Pi_n\bigl(W_n(\theta_0,\varepsilon)\bigr)}\,\int_{W_n(\theta_0,\varepsilon)}  R_\theta^{(n)}(X^{(n)})^{-{1\over 2}}\, \Pi_n(d\theta)\ \biggr)$$
$$\leq {\int_{W_n(\theta_0,\varepsilon)}  P_{\theta_0}^{(n)}R_\theta^{(n)}(X^{(n)})^{-{1\over 2}}\, \Pi_n(d\theta)\over e^{n\,\varepsilon^2\,({3\over 2}+c)} \Pi_n\bigl(W_n(\theta_0,\varepsilon)\bigr)},$$
where
$$P_{\theta_0}^{(n)}R_\theta^{(n)}(X^{(n)})^{-{1\over 2}}=1+\int \Bigl(\, \sqrt{p_{\theta_0}^{(n)}(X^{(n)})}-\sqrt{p_\theta^{(n)}(X^{(n)})}\,\Bigr)^2 \,{\sqrt{p_{\theta_0}^{(n)}(X^{(n)})}\over \sqrt{p_\theta^{(n)}(X^{(n)})}} \, \mu(dx)$$
$$+\int \Bigl(\, p_{\theta_0}^{(n)}(X^{(n)})-\sqrt{p_\theta^{(n)}(X^{(n)})\,p_{\theta_0}^{(n)}(X^{(n)})}\,\Bigr)\,  \mu(dx)$$
$$=1+\int \Bigl(\, \sqrt{p_{\theta_0}^{(n)}(X^{(n)})}-\sqrt{p_\theta^{(n)}(X^{(n)})}\,\Bigr)^2 \,{\sqrt{p_{\theta_0}^{(n)}(X^{(n)})}\over \sqrt{p_\theta^{(n)}(X^{(n)})}} \, \mu(dx)$$
$$+{1\over 2}\,\int \Bigl(\, \sqrt{p_{\theta_0}^{(n)}(X^{(n)})}-\sqrt{p_\theta^{(n)}(X^{(n)})}\,\Bigr)^2\,  \mu(dx)$$
$$= 1+ {3\over 2}\,H_*(p_{\theta_0}^{(n)},p_\theta^{(n)})^2\leq e^{{3\over 2}n\varepsilon^2},$$
which implies the required inequality and the proof of Lemma \ref{lem:2} is complete.
\end{proof}
\begin{lem}\label{lem:3}
Under Assumption \ref{ass:2}, the inequality
$$P_{\theta_0}^{(n)}\Big(\int_{\theta\in \Theta_1:\,d_n(\theta,\theta_0)\geq r\varepsilon}R_\theta^{(n)}(X^{(n)}) \Pi_n(d\theta)\Big)^\alpha$$
$$\leq K_1\sum\limits_{j=[r-1]}^\infty e^{-K_2nj^2\varepsilon^2}C(j\varepsilon,\{\theta\in \Theta_1:\, j\varepsilon < d_n(\theta,\theta_0)\leq 2j\varepsilon\},\alpha,e_n)^{K_3}$$
holds for all  $r\geq 2$, $\varepsilon>0$, $\Theta_1\subset \Theta$ and for all $n$ large enough.
\end{lem}

 \begin{proof} Note that $\{\theta\in \Theta_1:\,d_n(\theta,\theta_0)\geq r\varepsilon\}\subset \{\theta\in \Theta_1:\,d_n(\theta,\theta_0)\geq [r]\varepsilon\}=\cup_{j=[r-1]}^\infty \{\theta\in \Theta_1:\,j\varepsilon < d_n(\theta,\theta_0)\leq 2j\varepsilon\}:=\cup_{j=[r-1]}^\infty\Theta_{1,j}$. Using the inequality $(x+y)^\alpha\leq x^\alpha+y^\alpha$ for all $x,y\geq 0$ and Assumption \ref{ass:2} for $\Theta_1=\Theta_{1,j}$ we obtain
$$P_{\theta_0}^{(n)}\Big(\int_{\theta\in \Theta_1:\,d_n(\theta,\theta_0)\geq r\varepsilon}R_\theta^{(n)}(X^{(n)}) \Pi_n(d\theta)\Big)^\alpha$$
$$\leq \sum\limits_{j=[r-1]}^\infty P_{\theta_0}^{(n)}\Big(\int_{ \Theta_{1,j}}R_\theta^{(n)}(X^{(n)}) \Pi_n(d\theta)\Big)^\alpha$$
$$\leq K_1\sum\limits_{j=[r-1]}^\infty e^{-K_2nj^2\varepsilon^2}C(j\varepsilon,\{\theta\in \Theta_{1,j}:\, d_n(\theta,\theta_0)>j\varepsilon\},\alpha,e_n)^{K_3}$$
$$=K_1\sum\limits_{j=[r-1]}^\infty e^{-K_2nj^2\varepsilon^2}C(j\varepsilon,\{\theta\in \Theta_1:\, j\varepsilon < d_n(\theta,\theta_0)\leq 2j\varepsilon\},\alpha,e_n)^{K_3}.$$
The proof of Lemma \ref{lem:3} is complete.
\end{proof}

\begin{proof}[Proof of Theorem \ref{thm:1}.] Take a constant $c>1/c_0$. Then $e^{-n\,\varepsilon_n^2\, c}\leq e^{- cc_0\log n}
=1/n^{cc_0}$ and hence  $\sum_{n=1}^\infty e^{-n\,\varepsilon_n^2\, c}< \infty$. By Lemma \ref{lem:2} and the first Borel-Cantelli lemma, we get that for almost all $X^{(n)}$ the inequality
$$ \int_{\Theta} R_\theta^{(n)}(X^{(n)})\, \Pi_n(d\theta)\geq e^{-n\,\varepsilon_n^2\,(3+2c)}\ \Pi_n\bigl( W_n(\theta_0,\varepsilon_n)\bigr)$$
holds for all large $n$. Thus, for any $\delta>0$ we have
$$P_{\theta_0}^{(n)}\Bigl(\Pi_n\bigl(\theta\in \Theta_n:\,d_n(\theta,\theta_0)\geq r\,\varepsilon_n|X^{(n)}\bigr)\geq \delta \Bigr) $$
$$=P_{\theta_0}^{(n)}\Bigl(\delta^{-\alpha}\Pi_n\bigl(\theta\in \Theta_n:\,d_n(\theta,\theta_0)\geq r\,\varepsilon_n|X^{(n)}\bigr)^\alpha\geq 1 \Bigr) $$
$$\leq \delta^{-\alpha}P_{\theta_0}^{(n)}\Bigl(\Pi_n\bigl(\theta\in \Theta_n:\,d_n(\theta,\theta_0)\geq r\,\varepsilon_n|X^{(n)}\bigr)^\alpha \Bigr)\leq $$
$$ \delta^{-\alpha}e^{\alpha n\,\varepsilon_n^2\,(3+2c)}  \Pi_n\bigl(W_n(\theta_0,\varepsilon_n)\bigr)^{-\alpha} P_{\theta_0}^{(n)}\Big(\int_{\theta\in \Theta_n:\, d_n(\theta,\theta_0)\geq r\varepsilon_n}R_\theta^{(n)}(X^{(n)}) \Pi_n(d\theta)\Big)^\alpha,$$
which, by Lemma \ref{lem:3} and the inequality (2), does not exceed
$$K_1\delta^{-\alpha}e^{\alpha n\,\varepsilon_n^2\,(3+2c)} \sum\limits_{j=[r-1]}^\infty e^{-K_2nj^2\varepsilon_n^2+c_1j^2n\varepsilon_n^2}$$
$$\leq K_1\delta^{-\alpha}e^{\alpha n\,\varepsilon_n^2\,(3+2c)} \sum\limits_{j=[r-1]}^\infty e^{(c_1-K_2)jn\varepsilon_n^2}$$
$$=   {K_1e^{(c_1-K_2)[r-1]n\varepsilon_n^2+\alpha (3+2c) n\,\varepsilon_n^2}\over \delta^{\alpha} (1-e^{(c_1-K_2)n\varepsilon_n^2})}\leq {K_1n^{(c_1-K_2)[r-1]c_0+\alpha (3+2c) c_0}\over \delta^{\alpha} (1-n^{(c_1-K_2)c_0})}$$
$$\leq 2K_1\delta^{-\alpha}n^{(c_1-K_2)[r-1]c_0+\alpha (3+2c) c_0},$$
where the next last inequality holds for all large $r$ and the last inequality holds for all large $n$. Since the last exponent is strictly less than $-1$ for all large $r$, by the first Borel-Cantelli lemma we obtain that for almost all $X^{(n)}$,
$$\Pi_n\bigl(\theta\in \Theta_n:\,d_n(\theta,\theta_0)\geq r\,\varepsilon_n|X^{(n)}\bigr)\leq \delta $$
if $n$ is large enough, which yields the first assertion.

To get the second assertion, choose a positive constant $b$ with $c_2-{b\over 2}>{1\over c_0}$. We then follow the above proof, but take $c=c_2-{b\over 2}$ and $\delta=e^{-bn\varepsilon_n^2}$ instead, and note that
$$P_{\theta_0}^{(n)}\Bigl(\Pi_n\bigl(\theta\in \Theta:\,d_n(\theta,\theta_0)\geq r\,\varepsilon_n|X^{(n)}\bigr)\geq e^{-bn\varepsilon_n^2}\Bigr) $$
$$\leq P_{\theta_0}^{(n)}\Bigl(\Pi_n\bigl(\theta\in \Theta_n:\,d_n(\theta,\theta_0)\geq r\,\varepsilon_n|X^{(n)}\bigr)\geq {1\over 2}\,e^{-bn\varepsilon_n^2} \Bigr) $$
$$+P_{\theta_0}^{(n)}\Bigl(\Pi_n\bigl(\theta\in \Theta\setminus\Theta_n:\,d_n(\theta,\theta_0)\geq r\,\varepsilon_n|X^{(n)}\bigr)\geq {1\over 2}\,e^{-bn\varepsilon_n^2} \Bigr), $$
where by Lemma \ref{lem:2} the second term on the right hand side is dominated by
$${2e^{bn\varepsilon_n^2}e^{n\,\varepsilon_n^2\,(3+2c_2-b)} \over \Pi_n\bigl(W_n(\theta_0,\varepsilon_n)\bigr)} \,  P_{\theta_0}^{(n)}\int_{\theta\in \Theta\setminus \Theta_n:\, d_n(\theta,\theta_0)\geq r\varepsilon_n}R_\theta^{(n)}(X^{(n)}) \Pi_n(d\theta)$$
$$\leq {2e^{n\,\varepsilon_n^2\,(3+2c_2)}\over  \Pi_n\bigl(W_n(\theta_0,\varepsilon_n)\bigr)} \,  \int_{\Theta\setminus \Theta_n}P_{\theta_0}^{(n)}R_\theta^{(n)}(X^{(n)}) \Pi_n(d\theta)= {2e^{n\,\varepsilon_n^2\,(3+2c_2)} \Pi_n\bigl(\Theta\setminus \Theta_n\bigr)\over  \Pi_n\bigl(W_n(\theta_0,\varepsilon_n)\bigr)}.$$
Then, using the same argument as the above, one can easily prove the second assertion  and the proof of Theorem \ref{thm:1} is complete.
\end{proof}

Using the trivial inequality $C(\delta\varepsilon_n,\Theta_n,\alpha,e_n)\leq C(\varepsilon_n,\Theta_n,\alpha,e_n)$ for $\delta\geq 1$, one can similarly prove Theorem \ref{thm:2}. The proof of Theorem \ref{thm:3} is only a slight modification of the proof of Theorem \ref{thm:1} except that we need to apply Lemma 10 in Ghosal and van der Vaart \cite{ghv1}. The proof of Theorem \ref{thm:4.1} is completely similar to the proof of Theorem \ref{thm:1}, but instead of an application of Lemma \ref{lem:2} one needs the following Lemma.

\begin{lem}\label{lem:4} For independent observations $(X_1,X_2,\dots,X_n)$ we have
that the inequality
$$P_{\theta_0}^{(n)}\biggl(\ \int_{\Theta} \Bigl(\prod\limits_{i=1}^n{p_{\theta,i}(X_i)\over p_{\theta_0,i}(X_i)}\Bigr)^\beta\, \Pi_n(d\theta)\leq e^{-n\,\varepsilon^2\,(3+2c)\beta}\ \Pi_n\bigl( \overline{W}_n(\theta_0,\varepsilon) \bigr)\, \biggr)\leq  e^{-n\,\varepsilon^2\, c} $$
holds for all $n$, $\varepsilon>0$, $c>0$ and $0<\beta\leq 1$.
\end{lem}

\begin{proof}[Proof of Lemma \ref{lem:4}.] Similar to the proof of Lemma \ref{lem:2} one can get that
$$P_{\theta_0}^{(n)}\biggl(\ \int_{\Theta} \Bigl(\prod\limits_{i=1}^n{p_{\theta,i}(X_i)\over p_{\theta_0,i}(X_i)}\Bigr)^\beta\, \Pi_n(d\theta)\leq e^{-n\,\varepsilon^2\,(3+2c)\beta}\ \Pi_n\bigl( \overline{W}_n(\theta_0,\varepsilon) \bigr)\, \biggr)$$
$$\leq P_{\theta_0}^{(n)}\biggl(\ e^{n\,\varepsilon^2\,({3\over 2}+c)}\leq  \Bigl(\,{1\over \Pi_n\bigl(\overline{W}_n(\theta_0,\varepsilon)\bigr)}\,\int_{\overline{W}_n(\theta_0,\varepsilon)}  \Bigl(\prod\limits_{i=1}^n{p_{\theta,i}(X_i)\over p_{\theta_0,i}(X_i)}\Bigr)^\beta\, \Pi_n(d\theta)\Bigr)^{-{1\over 2\beta}}\ \biggr)$$
$$\leq P_{\theta_0}^{(n)}\biggl(\ e^{n\,\varepsilon^2\,({3\over 2}+c)}\leq  {1\over \Pi_n\bigl(\overline{W}_n(\theta_0,\varepsilon)\bigr)}\,\int_{\overline{W}_n(\theta_0,\varepsilon)}  \Bigl(\prod\limits_{i=1}^n{p_{\theta,i}(X_i)\over p_{\theta_0,i}(X_i)}\Bigr)^{-{1\over 2}}\, \Pi_n(d\theta)\ \biggr)$$
$$\leq {\int_{\overline{W}_n(\theta_0,\varepsilon)}  \prod_{i=1}^n \big(1+ {3\over 2}\,H_{*,i}(p_{\theta_1,i},p_{\theta_2,i})^2\big)\, \Pi_n(d\theta)\over e^{n\,\varepsilon^2\,({3\over 2}+c)} \Pi_n\bigl(\overline{W}_n(\theta_0,\varepsilon)\bigr)}$$
$$\leq {\int_{\overline{W}_n(\theta_0,\varepsilon)}  e^{{3\over 2}\,\sum_{i=1}^n H_{*,i}(p_{\theta_1,i},p_{\theta_2,i})^2}\, \Pi_n(d\theta)\over e^{n\,\varepsilon^2\,({3\over 2}+c)} \Pi_n\bigl(\overline{W}_n(\theta_0,\varepsilon)\bigr)}\leq  e^{-n\,\varepsilon^2\, c}, $$
which concludes the proof.
\end{proof}

\begin{proof}[Proof of Proposition \ref{prop:4}.] It is no restriction to assume that $n=2k$ is an even number. Similar to the proof of Proposition \ref{prop:2} we get that
the left side of the required inequality does not exceed $2C(\delta\,\varepsilon,\{\theta\in \Theta_1:\,d(\theta,\theta_0)> \varepsilon\},\alpha,d)$ times
$$\max_{1\leq j\leq N}\, P_{\theta_0}^{(n)}\bigg({{1\over \Pi_n(B_j)}\int_{B_j}{q_\theta(X_0)\prod\limits_{i=1}^{2k} p_\theta(X_i|X_{i-1})\over q_{\theta_0}(X_0)\prod\limits_{i=1}^{2k} p_{\theta_0}(X_i|X_{i-1}) } \, \Pi_n(d\theta) }\bigg)^\alpha$$
$$=\max_{1\leq j\leq N}\, P_{\theta_0}^{(n)}\bigg({\int_{B_j}q_\theta(X_0)\Pi_n(d\theta)\over q_{\theta_0}(X_0)\Pi_n(B_j)}\prod\limits_{s=0}^{2k-1}{I_{j,s}\over p_{\theta_0}(X_{s+1}|X_s)}\bigg)^\alpha=$$
$$\max_{1\leq j\leq N}\, P_{\theta_0}^{(n)}\Bigg(\bigg({\int_{B_j}q_\theta(X_0)\Pi_n(d\theta)\over q_{\theta_0}(X_0)\Pi_n(B_j)}\prod\limits_{t=1}^k{I_{j,2t-1}\over p_{\theta_0}(X_{2t}|X_{2t-1})}\bigg)^\alpha\bigg(\prod\limits_{t=0}^{k-1}{I_{j,2t}\over p_{\theta_0}(X_{2t+1}|X_{2t})}\bigg)^\alpha\Bigg)$$
$$\leq \max_{1\leq j\leq N}\, \Bigg(P_{\theta_0}^{(n)}\bigg({\int_{B_j}q_\theta(X_0)\Pi_n(d\theta)\over q_{\theta_0}(X_0)\Pi_n(B_j)}\prod\limits_{t=1}^k{I_{j,2t-1}\over p_{\theta_0}(X_{2t}|X_{2t-1})}\bigg)^{2\alpha}\Bigg)^{1\over 2}$$
$$\max_{1\leq j\leq N}\,\Bigg( P_{\theta_0}^{(n)}\bigg(\prod\limits_{t=0}^{k-1}{I_{j,2t}\over p_{\theta_0}(X_{2t+1}|X_{2t})}\bigg)^{2\alpha}\Bigg)^{1\over 2}:=\big(\max_{1\leq j\leq N}\,A_{j,k}\big)\ \big(\max_{1\leq j\leq N}\,B_{j,k}\big),$$
where the last inequality follows from H\"older's inequality, the set $B_j$ is defined in a similar way as that of Proposition \ref{prop:2} and we have used the notations $\prod\limits_{i=1}^0 p_\theta(X_i|X_{i-1})= 1$ and
$$I_{j,s}={\int_{B_j}q_\theta(X_0)\prod\limits_{i=1}^{s+1} p_\theta(X_i|X_{i-1})\, \Pi_n(d\theta)\over \int_{B_j}q_\theta(X_0)\prod\limits_{i=1}^s p_\theta(X_i|X_{i-1})\, \Pi_n(d\theta) }$$
for $s=0,1,\dots,2k-1$. We also let $I_{j,s}$ stand for the parameter of the corresponding integral means.
Take $\theta_j\in B_j$ for each $j$.
From Jensen's inequality and the assumption $a_0r(X_s)\leq p_{\theta}(X_s|X_{s-1})\leq  a_1 r(X_s)$ it turns out that
$$d(I_{j,s},\theta_j)^2=\int_{\mathfrak{X}}\int_{\mathfrak{X}}\big(\sqrt{I_{j,s}}
-\sqrt{p_{\theta_j}(X_{s+1}|X_s)}\,\big)^2 \,d\mu(X_{s+1})d\nu(X_s)$$
$$\leq   \int_{B_j}\int_{\mathfrak{X}}\int_{\mathfrak{X}}\big(\sqrt{p_\theta(X_{s+1}|X_s)}
-\sqrt{p_{\theta_j}(X_{s+1}|X_s)}\,\big)^2 \,d\mu(X_{s+1})$$
$${q_\theta(X_0)\prod\limits_{i=1}^s p_\theta(X_i|X_{i-1})\over \int_{B_j}q_\theta(X_0)\prod\limits_{i=1}^s p_\theta(X_i|X_{i-1})\, \Pi_n(d\theta) } \,  d\nu(X_s)\Pi_n(d\theta)$$
$$\leq {a_1\over a_0}  \int_{B_j}\int_{\mathfrak{X}}\int_{\mathfrak{X}}\big(\sqrt{p_\theta(X_{s+1}|X_s)}
-\sqrt{p_{\theta_j}(X_{s+1}|X_s)}\,\big)^2 \,d\mu(X_{s+1}) d\nu(X_s)$$
$${q_\theta(X_0)\prod\limits_{i=1}^{s-1} p_\theta(X_i|X_{i-1})\over \int_{B_j}q_\theta(X_0)\prod\limits_{i=1}^{s-1} p_\theta(X_i|X_{i-1})\, \Pi_n(d\theta) } \, \Pi_n(d\theta)\leq {4a_1\delta^2\varepsilon^2\over a_0}$$
$$= {a_1\over a_0}  \int_{B_j}d(\theta,\theta_j)^2{q_\theta(X_0)\prod\limits_{i=1}^{s-1} p_\theta(X_i|X_{i-1})\over \int_{B_j}q_\theta(X_0)\prod\limits_{i=1}^{s-1} p_\theta(X_i|X_{i-1})\, \Pi_n(d\theta) } \, \Pi_n(d\theta)\leq {4a_1\delta^2\varepsilon^2\over a_0}$$
Thus, $d(I_{j,s},\theta_j)\leq {2\sqrt{a_1}\delta\varepsilon\over \sqrt{a_0}}$ and
$d(I_{j,s},\theta_0)\geq d(\theta_j,\theta_0)-d(I_{j,s},\theta_j)\geq (1-{2\sqrt{a_1}\delta\over \sqrt{a_0}})\varepsilon$. Write
$$A_{j,k}^2=$$$$\int_{\mathfrak{X}^{2k-1}}\Bigg(\int_{\mathfrak{X}}\biggl( \int_{\mathfrak{X}}\big( {I_{j,2k-1}\over p_{\theta_0}(X_{2k}|X_{2k-1})}\big)^{2\alpha}\, d\mu(X_{2k})\biggr)\,p_{\theta_0}(X_{2k-1}|X_{2k-2})\,d\mu(X_{2k-1})\Bigg)$$
$$ \bigg({\int_{B_j}q_\theta(X_0)\Pi_n(d\theta)\over q_{\theta_0}(X_0)\Pi_n(B_j)}\prod\limits_{t=1}^{k-1}{I_{j,2t-1}\over p_{\theta_0}(X_{2t}|X_{2t-1})}\bigg)^{2\alpha}
$$$$q_{\theta_0}(X_0)\prod\limits_{s=0}^{2k-3}p_{\theta_0}(X_{s+1}|X_s)\,d\mu(X_0)d\mu(X_1)\dots d\mu(X_{2k-2}).$$
Take an nonnegative integer $m$ with ${2\alpha\over 1-2\alpha}\leq 2^m<{4\alpha\over 1-2\alpha}.$
Repeating the proof of Proposition \ref{prop:2} (applying the same procedure $m+1$ times instead of $m$ times) we get that
$$\int_{\mathfrak{X}}\biggl( \int_{\mathfrak{X}}\big( {I_{j,2k-1}\over p_{\theta_0}(X_{2k}|X_{2k-1})}\big)^{2\alpha}\, d\mu(X_{2k})\biggr)\,p_{\theta_0}(X_{2k-1}|X_{2k-2})\,d\mu(X_{2k-1})$$
$$\leq \int_{\mathfrak{X}}\biggl( 1-{1\over 2}\int_{\mathfrak{X}}\big( \sqrt{I_{j,2k-1}}-\sqrt{ p_{\theta_0}(X_{2k}|X_{2k-1})}\,\big)^2\, d\mu(X_{2k})\biggr)^{\alpha\over 2^{m-1}}$$$$p_{\theta_0}(X_{2k-1}|X_{2k-2})\,d\mu(X_{2k-1})\leq$$
$$ \biggl( 1-{1\over 2}\int_{\mathfrak{X}}\int_{\mathfrak{X}}\big( \sqrt{I_{j,2k-1}}-\sqrt{ p_{\theta_0}(X_{2k}|X_{2k-1})}\,\big)^2p_{\theta_0}(X_{2k-1}|X_{2k-2}) d\mu(X_{2k})d\mu(X_{2k-1})\biggr)^{\alpha\over 2^{m-1}}$$
$$\leq \biggl( 1-{a_0\over 2}\int_{\mathfrak{X}}\int_{\mathfrak{X}}\big( \sqrt{I_{j,2k-1}}-\sqrt{ p_{\theta_0}(X_{2k}|X_{2k-1})}\,\big)^2\, d\mu(X_{2k})d\nu(X_{2k-1})\biggr)^{{1\over 2}-\alpha}$$
$$= \Bigl( 1-{a_0 d(I_{j,2k-1},\theta_0)^2\over 2}\Bigr)^{{1\over 2}-\alpha}\leq e^{-(1-2\alpha)({\sqrt{a_0}\over 2}-\sqrt{a_1}\delta)^2 \varepsilon^2 }.$$
Hence we have
$$A_{j,k}^2\leq e^{-(1-2\alpha)({\sqrt{a_0}\over 2}-\sqrt{a_1}\delta)^2 \varepsilon^2 } \int_{\mathfrak{X}^{2k-1}}
 \bigg({\int_{B_j}q_\theta(X_0)\Pi_n(d\theta)\over q_{\theta_0}(X_0)\Pi_n(B_j)}\prod\limits_{t=1}^{k-1}{I_{j,2t-1}\over p_{\theta_0}(X_{2t}|X_{2t-1})}\bigg)^{2\alpha}
$$$$q_{\theta_0}(X_0)\prod\limits_{s=0}^{2k-3}p_{\theta_0}(X_{s+1}|X_s)\,d\mu(X_0)d\mu(X_1)\dots d\mu(X_{2k-2}).$$
Repeating the same argument $k-1$ times one can get that
$$A_{j,k}^2\leq e^{-(1-2\alpha)({\sqrt{a_0}\over 2}-\sqrt{a_1}\delta)^2 k\varepsilon^2 }\int_{\mathfrak{X}}
 \bigg({\int_{B_j}q_\theta(X_0)\Pi_n(d\theta)\over q_{\theta_0}(X_0)\Pi_n(B_j)}\bigg)^{2\alpha}q_{\theta_0}(X_0)\,d\mu(X_0)$$
$$\leq  e^{-(1-2\alpha)({\sqrt{a_0}\over 2}-\sqrt{a_1}\delta)^2 k\varepsilon^2 } \bigg(\int_{\mathfrak{X}}
{\int_{B_j}q_\theta(X_0)\Pi_n(d\theta)\over q_{\theta_0}(X_0)\Pi_n(B_j)}q_{\theta_0}(X_0)\,d\mu(X_0)\bigg)^{2\alpha}$$
$$= e^{-(1-2\alpha)({\sqrt{a_0}\over 2}-\sqrt{a_1}\delta)^2 k\varepsilon^2 }.$$
Similarly, we have
$$B_{j,k}^2\leq  e^{-(1-2\alpha)({\sqrt{a_0}\over 2}-\sqrt{a_1}\delta)^2 k\varepsilon^2 }.$$
Hence we have proved the required inequality and the proof of Proposition \ref{prop:4} is complete.
\end{proof}

The proof of Theorem \ref{thm:7} is completely similar to that of Theorem \ref{thm:1} except that we apply Proposition \ref{prop:4} and the following lemma.

\begin{lem}\label{lem:5} If there exists a constant $a_1\geq 1$ such that $\int_Ap_{\theta_0}(y|x)d\mu(y)\leq a_1 \int_A d\nu(y)$ for all $x\in \mathfrak{X}$ and $A\in{\cal A}$, then the inequality
$$P_{\theta_0}^{(n)}\biggl(\ \int_{\Theta} {q_\theta(X_0)\over q_{\theta_0}(X_0)}\prod\limits_{i=1}^n{p_\theta(X_i|X_{i-1})\over p_{\theta_0}(X_i|X_{i-1})}\, \Pi_n(d\theta)\leq e^{-n\,\varepsilon^2\,(3a_1+4c)}\ \Pi_n\bigl( W^1_n(\theta_0,\varepsilon) \bigr)\, \biggr)$$$$\leq  e^{-n\,\varepsilon^2\, c} $$
holds for all $n$, $\varepsilon>0$ and $c>0$.
\end{lem}

\begin{proof}[Proof of Lemma \ref{lem:5}.] Similar to the proof of Lemma \ref{lem:2} we have that the left hand side of the required inequality does not exceed
$${\int_{W^1_n(\theta_0,\varepsilon)}  P_{\theta_0}^{(n)}\Big({q_{\theta_0}(X_0)\over q_\theta(X_0)}\prod\limits_{i=1}^n{p_{\theta_0}(X_i|X_{i-1})\over p_\theta(X_i|X_{i-1})}\Big)^{1\over 4}\, \Pi_n(d\theta)\over e^{n\,\varepsilon^2\,({3a_1 \over 4}+c)} \Pi_n\bigl(W^1_n(\theta_0,\varepsilon)\bigr)}.$$
So it suffices to prove that $P_{\theta_0}^{(n)}\Big({q_{\theta_0}(X_0)\over q_\theta(X_0)}\prod\limits_{i=1}^n{p_{\theta_0}(X_i|X_{i-1})\over p_\theta(X_i|X_{i-1})}\Big)^{1\over 4}\leq e^{{3a_1\over 4 }n\varepsilon^2} $
for all $\theta\in W^1_n(\theta_0,\varepsilon)$. We assume without loss of generality that $n$ is an even number, say $n=2k$. Write
$${q_{\theta_0}(X_0)\over q_\theta(X_0)}\prod\limits_{i=1}^n{p_{\theta_0}(X_i|X_{i-1})\over p_\theta(X_i|X_{i-1})}={q_{\theta_0}(X_0)\over q_\theta(X_0)}\prod\limits_{j=1}^k{p_{\theta_0}(X_{2j}|X_{2j-1})\over p_\theta(X_{2j}|X_{2j-1})}  \prod\limits_{j=1}^k{p_{\theta_0}(X_{2j-1}|X_{2j-2})\over p_\theta(X_{2j-1}|X_{2j-2})}.$$
From H\"older's inequality it then turns out that
$$P_{\theta_0}^{(n)}\Big({q_{\theta_0}(X_0)\over q_\theta(X_0)}\prod\limits_{i=1}^n{p_{\theta_0}(X_i|X_{i-1})\over p_\theta(X_i|X_{i-1})}\Big)^{1\over 4}$$
$$\leq \biggl(P_{\theta_0}^{(n)}\Big({q_{\theta_0}(X_0)\over q_\theta(X_0)}\prod\limits_{j=1}^k{p_{\theta_0}(X_{2j}|X_{2j-1})\over p_\theta(X_{2j}|X_{2j-1})}\Big)^{1\over 2}\biggr)^{1\over 2}\,\biggl(P_{\theta_0}^{(n)}\Big(\prod\limits_{j=1}^k{p_{\theta_0}(X_{2j-1}|X_{2j-2})\over p_\theta(X_{2j-1}|X_{2j-2})}\Big)^{1\over 2}\biggr)^{1\over 2}$$
$$:=A_kB_k.$$
Hence by Fubini's theorem we get that $A_k^2 $ is equal to
$$\int_{\mathfrak{X}^{2k+1}} {q_{\theta_0}(X_0)^{3\over 2}\over q_\theta(X_0)^{1\over 2}}\prod\limits_{j=1}^k\bigg({p_{\theta_0}(X_{2j}|X_{2j-1})^{3\over 2}\over p_\theta(X_{2j}|X_{2j-1})^{1\over 2}}\,p_{\theta_0}(X_{2j-1}|X_{2j-2})\bigg)$$$$d\mu(X_0)d\mu(X_1)\dots d\mu(X_{2k})$$
$$=\int_{\mathfrak{X}^{2k-1}}\Bigg(\int_{\mathfrak{X}}\biggl( \int_{\mathfrak{X}} {p_{\theta_0}(X_{2k}|X_{2k-1})^{3\over 2}\over p_\theta(X_{2k}|X_{2k-1})^{1\over 2}}\, d\mu(X_{2k})\biggr)\,p_{\theta_0}(X_{2k-1}|X_{2k-2})\,d\mu(X_{2k-1})\Bigg)$$
$$ {q_{\theta_0}(X_0)^{3\over 2}\over q_\theta(X_0)^{1\over 2}}\prod\limits_{j=1}^{k-1}{p_{\theta_0}(X_{2j}|X_{2j-1})^{3\over 2}\over p_\theta(X_{2j}|X_{2j-1})^{1\over 2}}\,p_{\theta_0}(X_{2j-1}|X_{2j-2})\,d\mu(X_0)d\mu(X_1)\dots d\mu(X_{2k-2}),$$
where by the proof of Lemma 1 in Xing \cite{xi1} we have
$$\int_{\mathfrak{X}}\biggl( \int_{\mathfrak{X}} {p_{\theta_0}(X_{2k}|X_{2k-1})^{3\over 2}\over p_\theta(X_{2k}|X_{2k-1})^{1\over 2}}\, d\mu(X_{2k})\biggr)\,p_{\theta_0}(X_{2k-1}|X_{2k-2})\,d\mu(X_{2k-1})$$
$$=\int_{\mathfrak{X}}\Bigl(1+{3\over 2}H_* \big(p_{\theta_0}(\cdot|X_{2k-1}),p_{\theta}(\cdot|X_{2k-1})\big)^2\Bigr)\,p_{\theta_0}(X_{2k-1}|X_{2k-2})\,d\mu(X_{2k-1})$$
$$= 1+\int_{\mathfrak{X}}{3\over 2}H_* \big(p_{\theta_0}(\cdot|X_{2k-1}),p_{\theta}(\cdot|X_{2k-1})\big)^2\,p_{\theta_0}(X_{2k-1}|X_{2k-2})\,d\mu(X_{2k-1})$$
$$\leq 1+\int_{\mathfrak{X}}{3a_1\over 2}H_* \big(p_{\theta_0}(\cdot|X_{2k-1}),p_{\theta}(\cdot|X_{2k-1})\big)^2\,d\nu(X_{2k-1})$$
$$= 1+{3a_1\over 2}H_*(p_{\theta_0},p_{\theta})^2\leq e^{{3a_1\over 2}H_*(p_{\theta_0},p_{\theta})^2}.$$
Thus, we have obtained that $A_k\leq e^{{3a_1\over 4}H_*(p_{\theta_0},p_{\theta})^2} A_{k-1}$. Repeating the same argument $k-1$ times and using $a_1\geq 1$ one can get
$$A_k\leq e^{{3a_1\over 4}kH_*(p_{\theta_0},p_{\theta})^2}\Big( \int_{\mathfrak{X}}  {q_{\theta_0}(X_0)^{3\over 2}\over q_\theta(X_0)^{1\over 2}}\,d\mu(X_0)\Big)^{1\over 2}$$$$=e^{{3a_1\over 4}kH_*(p_{\theta_0},p_{\theta})^2}\bigl(1+{3\over 2}H_*(q_{\theta_0},q_{\theta})^2\bigr)^{1\over 2}\leq e^{{3\over 4}H_*(q_{\theta_0},q_{\theta})^2+{3a_1\over 4}k H_*(p_{\theta_0},p_{\theta})^2}.$$
Similarly, we can get that $B_k\leq e^{{3a_1\over 4}k H_*(p_{\theta_0},p_{\theta})^2}$. Therefore $A_kB_k\leq e^{{3\over 4}H_*(q_{\theta_0},q_{\theta})^2+{3a_1\over 4}n H_*(p_{\theta_0},p_{\theta})^2}\leq e^{{3a_1\over 4}n\varepsilon^2} $
for all $\theta\in W^1_n(\theta_0,\varepsilon)$, and the proof of Lemma \ref{lem:5} is complete.

\end{proof}

\end{document}